\documentclass[a4paper,10pt]{article}
\usepackage[left=4.1cm,right=4.1cm]{geometry}
\usepackage[utf8]{inputenc}
\usepackage{amsthm}
\usepackage{graphicx}
\usepackage{latexsym}

\usepackage[frak=euler,cal=pxtx,bb=ams]{mathalfa}
\usepackage{enumerate}
\usepackage[shortlabels]{enumitem}
\usepackage[T1]{fontenc}
\usepackage{titlesec}
\usepackage{hyperref}

\newif\iflight
\lighttrue
\newif\ifnocomment
\nocommenttrue
\newif\ifcolor

\titleformat{\chapter}[display]
{\LARGE\bfseries\scshape}
{{\chaptertitlename}\ \thechapter}
{.5em}
{\huge}

\titleformat{\section}[hang]
{\large\bfseries\scshape}
{\thesection.}{0.5em}{}

\setlist[enumerate]{leftmargin=7.5mm,itemsep=-1mm,topsep=0pt}

\newtheoremstyle{newdefinition}%
{\topsep}
{\topsep}
{\normalfont}
{0pt}
{\bfseries\scshape}
{.}
{ }
{}

\newtheoremstyle{newplain}%
{\topsep}
{\topsep}
{\itshape}
{0pt}
{\bfseries\scshape}
{.}
{ }
{}

\newtheoremstyle{remark}%
{\topsep}
{\topsep}
{\normalfont}
{0pt}
{\scshape}
{.}
{ }
{}

\makeatletter
\renewenvironment{proof}[1][\proofname]{\par
	\pushQED{\qed}%
	\normalfont \topsep6\p@\@plus6\p@\relax
	\trivlist
	\item[\hskip\labelsep
	\bfseries\scshape
	#1\@addpunct{.}]\ignorespaces
}{%
	\popQED\endtrivlist\@endpefalse
}
\makeatother

\let\oldbibliography\thebibliography
\renewcommand{\thebibliography}[1]{%
	\oldbibliography{#1}%
	\setlength{\itemsep}{0pt}%
}

\usepackage[labelfont=sc,bf]{caption}

\theoremstyle{newdefinition} \newtheorem{definition}{Definition}[section]
\theoremstyle{newdefinition} \newtheorem{remark}[definition]{Remark}
\theoremstyle{newplain} \newtheorem{theorem}[definition]{Theorem}
\theoremstyle{newplain} \newtheorem{lemma}[definition]{Lemma}
\theoremstyle{newplain} \newtheorem{fact}[definition]{Proposition}
\theoremstyle{newplain} \newtheorem{corollary}[definition]{Corollary}
\theoremstyle{newplain}
\theoremstyle{newplain}\newtheorem{observation}[definition]{Observation}
\RequirePackage[prologue]{xcolor}
\newcommand{\ray}[2]{(#1#2)^{\infty}}
\newcommand{\edgeinf}[6]
{\mathcal{E}_{\infty}(\ray{#1}{#2},\ray{#3}{#4},#5,#6)}
\newcommand{\edge}[4]
{\mathcal{E}(#1,#2,#3,#4)}
\newcommand{\bd}[1]{\partial_{\infty}#1}
\newcommand{\catz}{\mathtt{CAT(0)}}
\newcommand{\bdcatz}[1]{\partial_{\catz}#1}
\newcommand{\cay}[2]{\mathrm{Cay}(#1,#2)}
\newcommand{\intr}{\mathrm{int}}
\newcommand{\ka}{\mathrm{K}}
\newcommand{\katt}{\ka_{3,3}}
\newcommand{\kacz}{\ka_4}
\newcommand{\kapi}{\ka_5}
\newcommand{\dra}{\mathfrak{d}}
\newcommand{\square}{\Box}

\newcommand{\Cone}[1]{\mathrm{Cone}(#1)}

\newcommand{\triang}{\mathcal{T}}
\newcommand{\tra}[2]{\mathfrak{t}(#1,#2)}
\newcommand{\rys}[1]{\mathcal{R}_{#1}}
\newcommand{\gra}[1]{\mathcal{G}_{#1}}
\newcommand{\wiea}{{\Large $\bullet$}}
\newcommand{\sasa}{{\Large $\circ$}}
\newcommand{\wieb}{{\small $\spadesuit$}}
\newcommand{\sasb}{{\small $\heartsuit$}}

\newcommand{\en}{$\mathfrak{i}$}
\newcommand{\ky}{$\mathfrak{c}$}
\newcommand{\sg}{$\mathrm{SG}$}

\newcommand{\Z}{\mathbb{Z}}

\newcommand{\nb}[1]{\textcolor{blue}{\bf\large \#}\footnote{\textcolor{blue}{#1}}}
\newcommand{\nr}[1]{\textcolor{red}{\bf\large \#}\footnote{#1}}
\definecolor{darkgreen}{rgb}{0.05, 0.50, 0.06}
\newcommand{\ngr}[1]{\textcolor{darkgreen}{\bf\large \#}\footnote{\textcolor{darkgreen}{#1}}}
\definecolor{purp}{rgb}{0.53, 0.38, 0.56}
\newcommand{\npu}[1]{\textcolor{orange}{\bf\large
		\#}\footnote{#1}}
\newcommand{\nv}[1]{\textcolor{violet}{\bf\large
		\#}\footnote{\textcolor{violet}{#1}}}
\definecolor{afb}{rgb}{0.36, 0.54, 0.66}
\newcommand{\nby}[1]{\textcolor{afb}{\bf\large
		\#}\footnote{\textcolor{afb}{#1}}}
\definecolor{bro}{rgb}{0.59, 0.29, 0.0}
\newcommand{\nbz}[1]{\textcolor{bro}{\bf\large
		\#}\footnote{\textcolor{bro}{#1}}}
\definecolor{dmb}{rgb}{0.55, 0.0, 0.0}
\newcommand{\nce}[1]{\textcolor{dmb}{\bf\large
		\#}\footnote{\textcolor{dmb}{#1}}}

\ifnocomment
\renewcommand{\nb}[1]{}
\renewcommand{\nr}[1]{}
\renewcommand{\ngr}[1]{}
\renewcommand{\npu}[1]{}
\renewcommand{\nv}[1]{}
\renewcommand{\nby}[1]{}
\renewcommand{\nbz}[1]{}
\renewcommand{\nce}[1]{}
\fi

\newcommand{\cone}{black}
\newcommand{\ctwo}{gray}
\newcommand{\cthree}{blue}
\newcommand{\cfour}{orange}
\newcommand{\cfive}{darkgreen}
\newcommand{\csix}{dmb}

\ifcolor
\newcommand{\mone}[1]{#1}
\newcommand{\tone}[1]{\textbf{\textcolor{\cone}{#1}}}
\newcommand{\ttwo}[1]{\textcolor{\ctwo}{#1}}
\newcommand{\tthree}[1]{\textcolor{\cthree}{#1}}
\newcommand{\tfour}[1]{\textcolor{\cfour}{#1}}
\newcommand{\tfive}[1]{\textcolor{\cfive}{#1}}
\newcommand{\tsix}[1]{\textcolor{\csix}{#1}}
\else
\newcommand{\mone}[1]{#1}%
\newcommand{\tone}[1]{#1}
\newcommand{\ttwo}[1]{#1}
\newcommand{\tthree}[1]{#1}
\newcommand{\tfour}[1]{#1}
\newcommand{\tfive}[1]{#1}
\newcommand{\tsix}[1]{#1}
\fi

\ifcolor
\newcommand{\eone}{\tone{\textbf{(s1)}}}
\newcommand{\etwo}{\ttwo{\textbf{(s2)}}}
\newcommand{\ethree}{\tthree{\textbf{(s3)}}}
\newcommand{\efour}{\tfour{\textbf{(s4)}}}
\newcommand{\efive}{\tfive{\textbf{(s5)}}}
\newcommand{\esix}{\tsix{\textbf{(s6)}}}
\else
\newcommand{\eone}{\tone{(s1)}}
\newcommand{\etwo}{\ttwo{(s2)}}
\newcommand{\ethree}{\tthree{(s3)}}
\newcommand{\efour}{\tfour{(s4)}}
\newcommand{\efive}{\tfive{(s5)}}
\newcommand{\esix}{\tsix{(s6)}}
\fi

\usepackage{tikz}
\usetikzlibrary{calc}
\usetikzlibrary{decorations.pathreplacing}
\usetikzlibrary{decorations.pathmorphing}
\usetikzlibrary{decorations.text}
\usetikzlibrary{decorations.markings}
\usetikzlibrary{decorations.shapes}

\title{{\sc\bfseries Right-angled Coxeter groups\\ with Menger curve boundary}}
\author{Daniel Danielski\\{\footnotesize Daniel.Danielski@math.uni.wroc.pl}\\{\small Mathematical Institute, University of Wroc\l{}aw}\\{\small pl. Grunwaldzki 2/4 50-384 Wroc\l{}aw}}
\date{\vspace{-5ex}}

\begin{document}
\maketitle

\begin{abstract}
We find a sufficient condition for a nerve of a hyperbolic right-angled Coxeter group, under which the boundary of the group is homeomorphic to the Menger curve.
We show that this condition is satisfied by many triangulations of surfaces with boundary and other 2-complexes, as well as by some triangulations of disks $D^n$ for arbitrary $n\geq3$.
\end{abstract}

\section{Introduction}\label{c:wstep}
In this paper we address the question of when the boundary of a hyperbolic right-angled Coxeter group is homeomorphic to the Menger curve.
In \cite{DaHaWa19} the question 
is answered in the case when the nerve of such a group is a graph, and in \cite{HaHrSa19} the non-hyperbolic case is considered. 
Also, note that the Menger curve is the generic case for a Gromov boundary, \cite{DaGuP11}.

We show the following result, which describes a sufficient condition for a nerve of a right-angled Coxeter group, so that the group has the Menger curve as the boundary. 
We apply this result to families of nerves that are not graphs.
Non-standard terms appearing in the statement of this result are explained in Remark \ref{u:tmain}.

\begin{theorem}\label{t:main}
Let $N$ be the nerve of 
a hyperbolic right-angled Coxeter group $W_N$.
Assume that $N$ is inseparable, 
not a simplex,
\sg-non-planar,
and for each $n>1$ and any simplex $\Delta\subseteq N$ we have 
$H^n(N)=0$ and  $H^n(N\setminus\Delta)=0$. 
Then the boundary $\bd{W_N}$ is homeomorphic to the Menger curve.
\end{theorem}

\begin{remark}\label{u:tmain}
\begin{enumerate}[(i)]
\item Recall that the group $W_N$ is hyperbolic if and only if the nerve $N$ satisfies the no-$\square$ condition, i.e. it contains no cycle of length 4 as a full subcomplex, 
\cite[Theorem 17.1]{Moussong88}.

\item The space $N\setminus\Delta$ is obtained by removing  the \underline{closed} simplex $\Delta$ from $N$.  

\item The nerve $N$ is \emph{inseparable} if it is connected, has no separating pair of non-adjacent vertices, no separating simplex, and no separating full subcomplex
which is 
a suspension of a simplex. 

\item For a precise definition of \sg-non-planarity,  see Definition \ref{d:nieplan}.
A nerve $N$ is \sg-non-planar, e.g.,  when it has a full subcomplex which is obtained from the $\katt$ or $\kapi$ graph by subdividing each of its edges into at least two pieces.
\end{enumerate}
\end{remark}

To prove the above result, we
use similar methods to the ones used
in \cite{Swiat16},
replacing
the Whyburn's characterisation of the Sierpiński carpet 
by
the Anderson's characterisation of the Menger curve. 
The latter 
characterisation differs from the former by replacing the
planarity condition with the requirement that no open subset is planar.
To ensure that the latter requirement holds,
we
embed a 
non-planar graph into the boundary (and, assuming hyperbolicity, into
an
arbitrary open subset of the boundary) using some non-planar graphs contained in the nerve. 
In Subsection \ref{s:nieplanelem} we discuss some building blocks of such an embedding, which are then used
in the proof of Theorem \ref{t:nieplan}. 
The latter result is the
most original ingredient in 
the proof of Theorem \ref{t:main}.  

We apply Theorem \ref{t:main} to some families of nerves. 
The first class of examples are certain triangulations of surfaces with boundary.
We actually consider a wider class of 2-complexes, 
namely the ones without lonely edges (see Definition \ref{d:lonelyedge}).
We fully characterise those 2-complexes without lonely edges which admit a triangulation that is a nerve of a right-angled Coxeter group with Menger curve boundary (Theorem \ref{t:kompleksy}, Corollary \ref{c:powierzchnie}).
As 
another application, we show that the $n$-disk $D^n$ admits a triangulation that is a nerve of the right-angled Coxeter group with Menger curve boundary if and only if $n\geq3$ (Theorem \ref{t:dysk}).

\smallskip
\textbf{Organisation of the paper.} In Section \ref{c:prel} we introduce some notation and basic notions and we describe the case of the nerve being a cycle.
In Sections \ref{s:nieplanelem}--\ref{s:nieplandowod} we prove Theorem \ref{t:nieplan} concerning embeddings of graphs in boundaries.
In Section \ref{c:dowod} we prove Theorem \ref{t:main} and discuss necessity of its assumptions.
In Section \ref{c:prz} we show the above mentioned applications of Theorem \ref{t:main}.

\smallskip
\textbf{Acknowledgements.} The author would like to thank Jacek \'Swiątkowski for the introduction to the topic of this paper.
This research was partially supported by (Polish) Narodowe Centrum Nauki, grant UMO-2017/25/B/ST1/01335.

\section{Preliminaries}\label{c:prel}

In this section we introduce some basic notation and notions that are used in this paper. We will also recall some of their basic properties.
The reader may refer to the books \cite{HatBook01,DruKaBook18}.

\subsection{Right-angled Coxeter groups and their boundaries}

\begin{definition}\label{d:pgc}
Let $\Gamma=(V_{\Gamma},E_{\Gamma})$ be a graph. The \emph{right-angled Coxeter group} $W_{\Gamma}$ is the group given by the presentation $W_{\Gamma}:=\langle \{v:v\in V_{\Gamma}\}|\{v^2=1:v\in V_{\Gamma}\}\cup\{(uv)^2=1:(u,v)\in E_{\Gamma}\}\rangle$.
The \emph{nerve} $N_{\Gamma}$ of the group $W_{\Gamma}$ is a simplicial complex obtained by spanning a simplex on each full subgraph of $\Gamma$.

\end{definition}

\begin{remark}
By the definition we have a one-to-one correspondence between \emph{flag simplicial complexes} (i.e.~the ones having the property that each full subgraph of their 1-skeleton spans a simplex) and right-angled Coxeter groups.
\end{remark}

Our next goal is to define the Davis complex. 
In order to do this, we first discuss
special subgroups of right-angled Coxeter groups and define the Cayley graph.

\begin{definition} Let $W_N$ be a right-angled Coxeter group with nerve $N$ and let $T$ be a subset of the set of vertices of the complex $N$. 
The subgroup of $W_N$ generated by the set $T$ is called the \emph{special subgroup} of $W_N$ corresponding to the set $T$.
\end{definition}

\begin{remark}
\begin{enumerate}[(i)]
\item The special subgroup $G_T$ corresponding to the set $T$ is canonically isomorphic to the group $W_K$, where $K$ is the \emph{full subcomplex} of $N$ (i.e. the simplices of $N$ spanned on the vertices of $K$ are also simplices of $K$) having the set of vertices $T$, see \cite[Theorem 4.1.6(i)]{DavBook08}.

\item In particular, for a fixed nerve $N$, we have a one-to-one correspondence between full subcomplexes of $N$ and special subgroups of $W_N$.

\end{enumerate}
\end{remark}

\begin{definition}
	Let $G$ be a group with a set of generators $S$. The \emph{Cayley graph} $\cay{G}{S}$ is an undirected graph with the set of vertices $G$ and the set of edges $\{\{g,gs\}:g\in G,s\in S\}$. 
	We label the edge $\{g,gs\}$ with $s$.	
\end{definition}

\begin{remark}
	\begin{enumerate}[(i)]
\item In the remaining part of the paper we will consider only Cayley graphs of right-angled Coxeter groups $W_N$ with generating set $N^{(0)}$. In this case, we can see that each edge of the graph $\cay{W_N}{N^{(0)}}$ has exactly 1 label
and for each vertex $g$ of the graph $\cay{W_N}{N^{(0)}}$ and label $s\in N^{(0)}$ there is a unique edge of the graph $\cay{W_N}{N^{(0)}}$ labelled with $s$ having $g$ as one of its ends.

\item If $K$ is a full subcomplex of the nerve $N$, then the graph $\cay{W_K}{K^{(0)}}$ is a subgraph of $\cay{W_N}{N^{(0)}}$.

\item 
If $\Delta\subseteq N$ is a simplex, then
the graph $\cay{W_{\Delta}}{\Delta^{(0)}}$ is the 1-skeleton of a ($\dim\Delta +1$)-cube ($W_{\Delta}$ is isomorphic to $\Z_2^{\dim\Delta +1}$), 
for
each (left) coset of the subgroup  $W_{\Delta}$ of the group $W_N$ there 
is a corresponding
1-skeleton of a ($\dim\Delta +1$)-cube in the graph $\cay{W_N}{N^{(0)}}$. 	

\end{enumerate}
\end{remark}

Now we can state the definition of the Davis complex, which, owing to the above remarks, is 
well-defined.

\begin{definition}
Let $N$ be the nerve of the right-angled Coxeter group $W_N$.	
The \emph{Davis complex} $\Sigma_N$ is a cubical 
complex having $\cay{W_N}{N^{(0)}}$ as its 1-skeleton, in which for each simplex $\Delta\subseteq N$ we span a $(\dim\Delta+1)$-cube on each set of vertices of the graph $\cay{W_N}{N^{(0)}}$ corresponding to a left coset of the special subgroup $W_{\Delta}$.	
\end{definition}

\begin{remark}\label{u:komdav}
Let $W_N$ be a right-angled Coxeter group with nerve $N$.	

\begin{enumerate}[(i)]
\item The natural action of the group $W_N$ on its Cayley graph $\cay{W_N}{N^{(0)}}$ can be extended to an action by automorphisms on the whole Davis complex $\Sigma_N$.

\item The link of each vertex of the complex $\Sigma_N$ is 
isomorphic to the nerve $N$. Moreover, the labels of the vertices of $N$ are the same as the labels of the corresponding 
edges in the complex $\Sigma_N$. 

\item If $K$ is a full subcomplex of $N$, then $\Sigma_K\subseteq\Sigma_N$.

\item The cubical complex $\Sigma_N$ 
has a 
natural piecewise euclidean $\catz$ metric,
\cite{Moussong88}.
This metric is given by taking the euclidean metric of a unit cube on each of the cubes and extending it to the whole complex by taking the infima of the lengths of 
chains of segments 
such that each of these segments is contained in a single cube.
The reader is referred to \cite{BriHaeBook99} for more information about $\catz$ geometry.

\end{enumerate}
\end{remark}

Now we define the boundary of a right-angled Coxeter group.

\begin{definition}
\begin{enumerate}[(i)]
\item Let $X$ be a $\catz$ space. The \emph{$\catz$ boundary} (also known as the \emph{visual boundary}) $\bdcatz{X}$ of the space $X$ is the space of geodesic rays starting at some fixed point $x_0$, with the topology of the inverse system $(\{S_R:R>0\},\{\pi^R_r:R>r>0\})$, where $S_R$ are the points at distance $R$ from $x_0$, 
and $\pi^R_r$ is the natural projection from $S_R$ onto $S_r$ (mapping the point $x$ of the larger sphere to the unique point $x'$ of the smaller sphere lying on the geodesic that joins $x_0$ with $x$).

\item Let $W_N$ be a right-angled Coxeter group. The \emph{boundary} 
of 
$W_N$ is the space $\bd{W_N}:=\bdcatz{\Sigma_N}$.
\end{enumerate}
\end{definition}

\begin{remark}\label{u:wlbrzeg}
\begin{enumerate}[(i)]
\item The $\catz$ boundary (up to a natural homeomorphism) is independent of the choice of the origin of the geodesic rays, \cite[Section I.8]{DavBook08}.
In the 
remainder 
of the paper we will consider Davis complexes with the base point at the vertex corresponding to the 
identity
element, and for a geodesic ray $\varrho$ we will denote by $[\varrho]$ the corresponding point of the boundary. 

\item The boundary of any right-angled Coxeter group is metrisable (as an inverse limit of metric spaces) and compact, \cite[Section I.8]{DavBook08}.
\label{u:wlbrzegzwme} 

\item 
In the case when the group $W_N$ is hyperbolic, its boundary $\bd{W_N}$ is homeomorphic to its Gromov boundary, see \cite[Chapter III.H.3]{BriHaeBook99} for more details. 
\end{enumerate}
\end{remark}

The following is a folklore result, see \cite[Appendix]{Swiat16} for a proof.

\begin{fact}\label{f:brzegpodkpl}
Let $K$ be a full subcomplex  of the nerve $N$ of the group $W_N$. Then
\begin{enumerate}[(i)]
	\item The complex $\Sigma_K$ is a convex subcomplex of  $\Sigma_N$.
	
	\item The boundary $\bd{W_K}$ is a subspace of the boundary $\bd{W_N}$.
\end{enumerate}
\end{fact}

\subsection{Groups with a cycle as a nerve}

In this section we give an example that illustrates the definitions from the previous section and is important later in this paper.

\begin{figure}
	\centering
\iflight	
	\centerline{\begin{tikzpicture}[scale=0.65]
\ifcolor
\newcommand{\gcol}{red}
\else
\newcommand{\gcol}{gray}
\fi
\newcommand{\gcolcoeff}{70}

\begin{scope}[xshift=-70]
\begin{scope}[rotate=-90]
\newcommand{\ang}{72}

\filldraw (0,0) circle (2pt);

\newcommand{\rone}{2}

\foreach \x in {1,2,3,4,5}
{
	\filldraw (\x*\ang:\rone) circle (2pt);
	\draw (0,0) -- (\x*\ang:\rone);		
}

\newcommand{\rtwo}{3.2}
\newcommand{\offstwo}{\ang*0.5}

\foreach \x in {1,2,3,4,5}
{
	\filldraw (\x*\ang+\offstwo:\rtwo) circle (2pt);
	\draw (\x*\ang:\rone) -- (\x*\ang+\offstwo:\rtwo) -- (\x*\ang+\ang:\rone);
}

\newcommand{\rthree}{4}
\newcommand{\devthree}{10}

\foreach \x in {1,2,3,4,5}
{
	\filldraw (\x*\ang-\devthree:\rthree) circle (2pt);
	\filldraw (\x*\ang+\devthree:\rthree) circle (2pt);
	\draw (\x*\ang-\devthree:\rthree) -- (\x*\ang:\rone) -- (\x*\ang+\devthree:\rthree);
}

\newcommand{\rfour}{5}
\newcommand{\devfour}{10}

\foreach \x in {1,2,3,4,5}
{
	\filldraw (\x*\ang+\offstwo-\devfour:\rfour) circle (2pt);
	\filldraw (\x*\ang+\offstwo:\rfour) circle (2pt);
	\filldraw (\x*\ang+\offstwo+\devfour:\rfour) circle (2pt);
	
	\draw (\x*\ang+\devthree:\rthree) -- (\x*\ang+\offstwo-\devfour:\rfour) -- (\x*\ang+\offstwo:\rtwo) -- (\x*\ang+\offstwo:\rfour) -- (\x*\ang+\offstwo:\rtwo) -- (\x*\ang+\offstwo+\devfour:\rfour) -- (\x*\ang+\ang-\devthree:\rthree); 
}

\foreach \x in {1,2,3,4,5}
{
	\filldraw (\x*\ang:\rfour) circle (2pt);
	\draw (\x*\ang-\devthree:\rthree) -- (\x*\ang:\rfour) -- (\x*\ang+\devthree:\rthree);
}

\newcommand{\rsix}{6}
\newcommand{\devsix}{\devfour*0.5}

\foreach \x in {1,2,3,4,5}
{
	\filldraw (\x*\ang+\offstwo-\devsix:\rsix) circle (2pt);
	\filldraw (\x*\ang+\offstwo+\devsix:\rsix) circle (2pt);
	 \draw (\x*\ang+\offstwo-\devfour:\rfour) -- (\x*\ang+\offstwo-\devsix:\rsix) -- (\x*\ang+\offstwo:\rfour) -- (\x*\ang+\offstwo+\devsix:\rsix) -- (\x*\ang+\offstwo+\devfour:\rfour);
}


\foreach \x in {1,2,3,4,5}
{
	\draw[color=\gcol,thin] (0,0) -- (\x*\ang+\offstwo:\rtwo);
	\foreach \y in {1,2,3,4}
	{
		\draw[color=\gcol,thin] (0,0) -- ($(\x*\ang+\offstwo:\rtwo)+\y*0.25*(\x*\ang:\rone)-\y*0.25*(\x*\ang+\offstwo:\rtwo)$);	
	}
	\draw[color=\gcol,thin] (0,0) -- (\x*\ang:\rone);
	\foreach \y in {0,1,2,3}
	{
		\draw[color=\gcol,thin] (0,0) -- ($(\x*\ang+\offstwo:\rtwo)+\y*0.25*(\x*\ang+\ang:\rone)-\y*0.25*(\x*\ang+\offstwo:\rtwo)$);	
	}
	\draw[color=\gcol,thin] (0,0) -- (\x*\ang+\ang+\offstwo:\rtwo);	
}

\foreach \x in {1,2,3,4,5}
{
	\draw[color=\gcol!\gcolcoeff!black,thin] (0,0) -- (\x*\ang:\rone);
}

\foreach \x in {0,1,2}
{
	\draw[color=\gcol] (3*\ang:\rone) -- ($(3*\ang:\rfour)-0.5*\x*(3*\ang:\rfour)+0.5*\x*(3*\ang-\devthree:\rthree)$);	
	\draw[color=\gcol] (3*\ang:\rone) -- ($(3*\ang:\rfour)-0.5*\x*(3*\ang:\rfour)+0.5*\x*(3*\ang+\devthree:\rthree)$);
}

\foreach \x in {0,1,2}
{
	\draw[color=\gcol] (2*\ang+\offstwo:\rtwo) -- ($(2*\ang+\offstwo:\rfour)-0.5*\x*(2*\ang+\offstwo:\rfour)+0.5*\x*(2*\ang+\offstwo-\devsix:\rsix)$);	
	\draw[color=\gcol] (2*\ang+\offstwo:\rtwo) -- ($(2*\ang+\offstwo:\rfour)-0.5*\x*(2*\ang+\offstwo:\rfour)+0.5*\x*(2*\ang+\offstwo+\devsix:\rsix)$);	
}

\draw[color=\gcol] ($0.75*(3*\ang:\rone)+0.25*(3*\ang+\offstwo:\rtwo)$) -- ($0.5*(3*\ang+\devthree:\rthree)+0.5*(3*\ang+\offstwo-\devfour:\rfour)$); 
\draw[color=\gcol] ($0.5*(3*\ang:\rone)+0.5*(3*\ang+\offstwo:\rtwo)$) -- (3*\ang+\offstwo-\devfour:\rfour);
\draw[color=\gcol] ($0.25*(3*\ang:\rone)+0.75*(3*\ang+\offstwo:\rtwo)$) -- ($0.5*(3*\ang+\offstwo:\rtwo)+0.5*(3*\ang+\offstwo-\devfour:\rfour)$) -- ($0.5*(3*\ang+\offstwo-\devsix:\rsix)+0.5*(3*\ang+\offstwo-\devfour:\rfour)$);

\draw[color=\gcol] ($0.75*(3*\ang:\rone)+0.25*(3*\ang-\offstwo:\rtwo)$) -- ($0.5*(3*\ang-\devthree:\rthree)+0.5*(3*\ang-\offstwo+\devfour:\rfour)$); 
\draw[color=\gcol] ($0.5*(3*\ang:\rone)+0.5*(3*\ang-\offstwo:\rtwo)$) -- (3*\ang-\offstwo+\devfour:\rfour);
\draw[color=\gcol] ($0.25*(3*\ang:\rone)+0.75*(3*\ang-\offstwo:\rtwo)$) -- ($0.5*(3*\ang-\offstwo:\rtwo)+0.5*(3*\ang-\offstwo+\devfour:\rfour)$) -- ($0.5*(3*\ang-\offstwo+\devsix:\rsix)+0.5*(3*\ang-\offstwo+\devfour:\rfour)$);

\foreach \x in {2}
{
	\draw[color=\gcol!\gcolcoeff!black] (3*\ang:\rone) -- ($(3*\ang:\rfour)-0.5*\x*(3*\ang:\rfour)+0.5*\x*(3*\ang-\devthree:\rthree)$);	
	\draw[color=\gcol!\gcolcoeff!black] (3*\ang:\rone) -- ($(3*\ang:\rfour)-0.5*\x*(3*\ang:\rfour)+0.5*\x*(3*\ang+\devthree:\rthree)$);
}

\foreach \x in {0}
{
	\draw[color=\gcol!\gcolcoeff!black] (2*\ang+\offstwo:\rtwo) -- ($(2*\ang+\offstwo:\rfour)-0.5*\x*(2*\ang+\offstwo:\rfour)+0.5*\x*(2*\ang+\offstwo-\devsix:\rsix)$);	
	\draw[color=\gcol!\gcolcoeff!black] (2*\ang+\offstwo:\rtwo) -- ($(2*\ang+\offstwo:\rfour)-0.5*\x*(2*\ang+\offstwo:\rfour)+0.5*\x*(2*\ang+\offstwo+\devsix:\rsix)$);	
}
\end{scope}

\draw[->,decorate,decoration={snake},very thick] (5,0) -- (7.5,0);

\end{scope}

\begin{scope}[xshift=300]
\newcommand{\ronea}{2.2}
\newcommand{\roneb}{2.7}
\newcommand{\ang}{24}
\newcommand{\angdev}{3}

\draw (0,0) circle (\ronea);
\draw (0,0) circle (\roneb);

\foreach \x in {1,2,3,4,5,6,7,8,9,10,11,12,13,14,15}
{
	\foreach \y in {-2,-1,0,1,2}
	{
		\draw[color=\gcol!85!black] (\x*\ang:\ronea) -- (\x*\ang+\y*\angdev:\roneb);
	}
	\foreach \y in {1,2,3}
	{
		\draw[color=\gcol] (\x*\ang+\ang-0.25*\y*\ang:\ronea) -- (0.25*\y*\x*\ang+0.25*\y*2*\angdev+\x*\ang+\ang-2*\angdev-0.25*\y*\x*\ang-0.25*\y*\ang+0.25*\y*2*\angdev:\roneb);
	}
}

\newcommand{\sele}{10pt};
\newcommand{\amp}{3pt}
\newcommand{\rout}{5}

\draw[thick,color=\gcol!\gcolcoeff!black,rounded corners,decorate,decoration={random steps,segment length=\sele,amplitude=\amp}]  (0,0) -- (\ang*7+0.5*\ang:\ronea);
\draw[thick,color=\gcol!\gcolcoeff!black] (\ang*7+0.5*\ang:\ronea) --  (\ang*7+0.5*\ang:\roneb);
\draw[thick,color=\gcol!\gcolcoeff!black,rounded corners,decorate,decoration={random steps,segment length=\sele,amplitude=\amp}]  (\ang*7+0.5*\ang:\roneb) -- (\ang*7+0.5*\ang:\rout);

\draw[thick,color=\gcol!\gcolcoeff!black,rounded corners,decorate,decoration={random steps,segment length=\sele,amplitude=\amp}]  (0,0) -- (\ang*11:\ronea);
\draw[thick,color=\gcol!\gcolcoeff!black] (\ang*11:\ronea) --  (\ang*11-\angdev:\roneb);
\draw[thick,color=\gcol!\gcolcoeff!black,rounded corners,decorate,decoration={random steps,segment length=\sele,amplitude=\amp}]  (\ang*11-\angdev:\roneb) -- (\ang*11-\angdev:\rout);

\draw[thick,color=\gcol!\gcolcoeff!black,rounded corners,decorate,decoration={random steps,segment length=\sele,amplitude=\amp}]  (0,0) -- (\ang*15+0.75*\ang:\ronea);
\draw[thick,color=\gcol!\gcolcoeff!black] (\ang*15+0.75*\ang:\ronea) --  (\ang*15+0.5*\ang+0.125*\ang-0.125*\angdev:\roneb);
\draw[thick,color=\gcol!\gcolcoeff!black,rounded corners,decorate,decoration={random steps,segment length=\sele,amplitude=\amp}]  (\ang*15+0.5*\ang+0.125*\ang-0.125*\angdev:\roneb) -- (\ang*15+0.5*\ang+0.125*\ang-0.125*\angdev:\rout) node[midway,below] {\textcolor{black}{{$\varrho$}}};

\filldraw (0,0) circle (2pt);
\draw[thick] (0,0) circle (\rout);

\filldraw[color=\gcol!\gcolcoeff!black] 
(\ang*7+0.5*\ang:\rout) circle (2pt);
\filldraw[color=\gcol!\gcolcoeff!black] (\ang*11-\angdev:\rout) circle (2pt);
\filldraw[color=\gcol!\gcolcoeff!black] (\ang*15+0.5*\ang+0.125*\ang-0.125*\angdev:\rout) circle (2pt) node[right] {\textcolor{black}{{$[\varrho]$}}};

\path (90:\rout*1.1) node {\large {$\bd{W_C}$}};
\path (120:\rout*0.75) node {\Large {$\Sigma_C$}};
\path (60:\ronea*0.87) node {\footnotesize {$S_r$}};
\path (60:\roneb*1.13) node {\footnotesize {$S_{r+\varepsilon}$}};

\end{scope}

\end{tikzpicture}}
\else
	\includegraphics[width=0.9\textwidth]{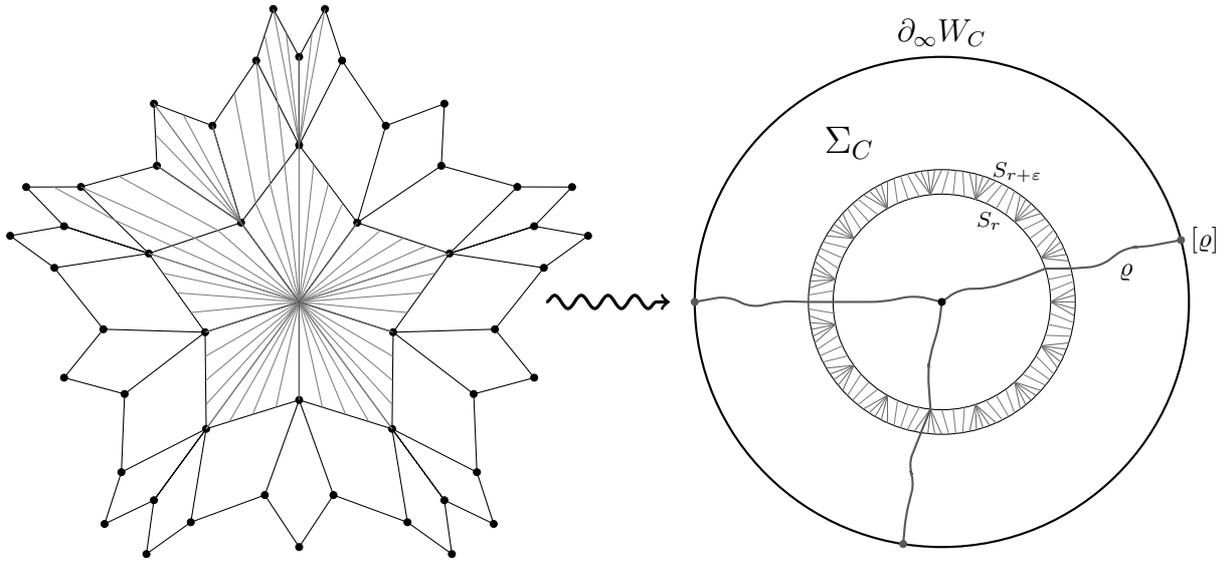}
\fi
	
	\caption{\emph{Left:} part of the complex $\Sigma_C$ with $n=5$ with some geodesics. \emph{Right:} some way of ``drawing'' $\Sigma_C\cup\bd{W_C}$ on the disk $D^2$ that illustrates the conclusion of Proposition \ref{f:brzegcyklu}.}
	\label{fig:brzegcyklu}
\end{figure}

Let $C$ be an $n$-cycle for $n\geq4$. 
Then $C$ is a flag simplicial complex, and therefore it is the nerve of the right-angled Coxeter group $W_C$. 
We argue that the Davis complex $\Sigma_C$ 
is homeomorphic to a tiling of the interior $\intr D^2$ of the disk $D^2$, such that each vertex is of degree $n$, with points in the boundary $\partial D^2$ corresponding to
the points in the boundary $\bdcatz{\Sigma_C}$
(see Proposition \ref{f:brzegcyklu}) 

We analyse the geodesic rays based at the vertex $e\in\Sigma_C$. By Proposition \ref{f:brzegpodkpl}, for each 2-cube of the complex $\Sigma_C$, the metric induced from $\Sigma_C$ is the standard Euclidean metric. In particular, the geodesic rays are chains of segments with each segment contained in a 2-cube and geodesic rays starting at $e$ go radially to the boundary of the union of the 2-cubes that contain $e$.
By shadow 
characterisation of bifurcations of geodesics,
\cite[Lemma 2d.1]{DavJan91},
a geodesic can be extended if and only if the angle between it and its continuation is at least $2\pi$ in both possible measuring directions. In particular, a geodesic ending in the interior of a 1-cell can be extended in a unique way, and a geodesic ending at a vertex can be extended in directions spanning an angle $(n-4)\pi$, in particular, we have a bifurcation iff $n\geq5$.
See Figure \ref{fig:brzegcyklu}.
Now we proceed to the description of the inverse system $(S_R,\pi_r^R)$, that 
appears
in the definition of the boundary $\bd{W_C}$. 
Each geodesic can be extended to a geodesic ray and the geodesic rays (based at the vertex $e$) cover the whole complex $\Sigma_C$. 
Furthermore, looking at the local behaviour of the geodesics, the spaces $\{S_R:R>0\}$ correspond to concentrically embedded copies of the circle $S^1$ and the projections $\pi_r^R$ are monotonic (i.e. preimage of each point is connected). 
Thus we can view the complex $\Sigma_C$ as an inverse system, which can be ``placed'' on a plane, and, furthermore, we have the following fact.

\begin{fact}\label{f:brzegcyklu}
There exist homeomorphisms $h_1:\Sigma_C\to\intr D^2$ and $h_2:\bd{W_C}\to\partial D^2$ such that for each geodesic ray $\varrho$ in $\Sigma_C$ that starts at the vertex $e$ 
we have
$\overline{h_1(\varrho)}\cap\partial D^2=\{h_2([\varrho])\}$. 		
\end{fact}

We skip the proof of the above proposition as it is a well known folklore fact. It can be also derived elementarily, using the observations made in the discussion above its statement.
Proposition \ref{f:brzegcyklu} will be used throughout the course of this paper and we will not refer to it explicitly.

\subsection{Menger curve}

In this paper, we use the following characterisation of the Menger curve due to Anderson \cite{An58a,An58b}.

\begin{fact}\label{f:charmenger}
Each topological space that is metrisable, compact, 1-dimen-sional, connected, locally connected, has no local cut-points and has no open planar subsets is homeomorphic to the Menger curve.
\end{fact}

\section{Non-planarity and proof of the main theorem}\label{c:nieplan}
In this section we find a sufficient condition for non-planarity of the boundary of a right-angled Coxeter group.
More precisely, we show how to embed some graphs 
into the boundary.
If the embedded graph is non-planar, then, under the additional assumption that the group itself is hyperbolic, no open subset of the boundary is planar (see the last part of the proof of Theorem \ref{t:main} in Subsection \ref{c:dowod}). 

\subsection{Building blocks}\label{s:nieplanelem}

By Proposition \ref{f:brzegpodkpl}, for each cycle 
that is
a full subcomplex of the nerve $N$, there is a corresponding homeomorphic copy of the circle $S^1$ in the boundary $\bd{W_N}$. The arcs of the embeddings of graphs that we construct in this section will consist of parts of such circles.

Next, observe that 
the special subgroup $W_{\{a,b\}}$ for any 2 distinct vertices $a,b$ of $N$ not connected by an edge
is the infinite dihedral group, whose Davis complex is the real line subdivided into 1-cubes labelled alternatingly by $a$ and $b$. Therefore, its boundary consists of two points corresponding to the geodesic rays $ababab\ldots$ and $bababa\ldots$. We denote these points $\ray{a}{b}$ and $\ray{b}{a}$, respectively. 
By Proposition \ref{f:brzegpodkpl}, each pair of non-adjacent vertices in the nerve $N$ gives rise to a pair of points in the boundary $\bd{W_N}$. 
The vertices of the embeddings of graphs that we construct in this section will consist of such points.

\begin{figure}
\centering
	\centerline{\begin{tikzpicture}
\newcommand{\acol}{darkgray}
\newcommand{\ang}{60}
\newcommand{\scloc}{1.2}

\begin{scope}[xshift=-140,yshift=50]
\begin{scope}[rotate=-90,scale=0.25]
\newcommand{\rone}{5}

\draw (0,0) circle (\rone);
\foreach \x/\n in {1/$x_5$,2/$x_6$,3/$x_1$,4/$x_2$,5/$x_3$,6/$x_4$}
{
	\filldraw (\x*\ang:\rone) circle (6pt);
	\path (\x*\ang:\rone*1.18) node {\footnotesize {\n}};
}
\draw[color=\acol,very thick] (4*\ang:\rone) arc (4*\ang:6*\ang:\rone);
\end{scope}
\end{scope}

\begin{scope}[scale=0.65,rotate=-90]
\filldraw (0,0) circle (2pt);

\newcommand{\rone}{1*\scloc}

\foreach \x in {1,2,3,4,5,6}
{
	\filldraw (\x*\ang:\rone) circle (2pt);
	\draw (0,0) -- (\x*\ang:\rone);	
}

\newcommand{\rtwo}{1.5*\scloc}
\newcommand{\offstwo}{\ang*0.5}

\foreach \x in {1,2,3,4,5,6}
{
	\filldraw (\x*\ang+\offstwo:\rtwo) circle (2pt);
	\draw (\x*\ang:\rone) -- (\x*\ang+\offstwo:\rtwo) -- (\x*\ang+\ang:\rone);
}

\newcommand{\rthree}{2.0*\scloc}
\newcommand{\devthree}{15}

\foreach \x in {3}
{
	\filldraw (\x*\ang-\devthree:\rthree) circle (2pt);
	\filldraw (\x*\ang:\rthree) circle (2pt);
	\filldraw (\x*\ang+\devthree:\rthree) circle (2pt);
	\draw (\x*\ang-\devthree:\rthree) -- (\x*\ang:\rone);
	\draw (\x*\ang+\devthree:\rthree) -- (\x*\ang:\rone);
	\draw (\x*\ang:\rone) -- (\x*\ang:\rthree);
}

\newcommand{\rfour}{2.5*\scloc}
\newcommand{\devfour}{0}

\foreach \x in {2}
{
	\filldraw (\x*\ang+\offstwo+\devfour:\rfour) circle (2pt);
	
	\draw (\x*\ang+\offstwo:\rtwo) -- (\x*\ang+\offstwo+\devfour:\rfour) -- (\x*\ang+\ang-\devthree:\rthree); 
}

\foreach \x in {3}
{
	\filldraw (\x*\ang+\offstwo-\devfour:\rfour) circle (2pt);
	
	\draw (\x*\ang+\devthree:\rthree) -- (\x*\ang+\offstwo-\devfour:\rfour) -- (\x*\ang+\offstwo:\rtwo); 
}

\newcommand{\devfourmid}{10}

\foreach \x in {3}
{
	\filldraw (\x*\ang-\devfourmid:\rfour) circle (2pt);
	\filldraw (\x*\ang+\devfourmid:\rfour) circle (2pt);
	\draw (\x*\ang-\devthree:\rthree) -- (\x*\ang-\devfourmid:\rfour) -- (\x*\ang:\rthree) -- (\x*\ang+\devfourmid:\rfour) -- (\x*\ang+\devthree:\rthree);
}

\ifcolor
\newcommand{\gcol}{red}
\else
\newcommand{\gcol}{gray}
\fi
\newcommand{\gcolcoeff}{70}
\newcommand{\rout}{5}
\newcommand{\geodev}{15};
\newcommand{\sele}{10pt};
\newcommand{\amp}{3pt}

\newcommand{\geowta}{0.05}
\newcommand{\geowtb}{(1-\geowta)}
\newcommand{\geodevmult}{1.5}

\foreach \x/\n in {1/$x_5$,2/$x_6$,4/$x_2$,5/$x_3$,6/$x_4$}
{
	\draw[color=\gcol!\gcolcoeff!black,thick] (0,0) -- (\x*\ang:\rone) node[midway] {\footnotesize \textcolor{black}{{\n}}};
	
	
	\draw[color=\gcol] (\x*\ang:\rone) -- ($\geowtb*(\x*\ang:\rone)+\geowta*(\x*\ang-\geodev*\geodevmult:\rout)$);
	\draw[color=\gcol,rounded corners,decorate,decoration={random steps,segment length=\sele,amplitude=\amp}] ($\geowtb*(\x*\ang:\rone)+\geowta*(\x*\ang-\geodev*\geodevmult:\rout)$) -- (\x*\ang-\geodev:\rout);
			
	\draw[color=\gcol] (\x*\ang:\rone) -- ($\geowtb*(\x*\ang:\rone)+\geowta*(\x*\ang:\rout)$);
	\draw[color=\gcol,rounded corners,decorate,decoration={random steps,segment length=\sele,amplitude=\amp}] ($\geowtb*(\x*\ang:\rone)+\geowta*(\x*\ang:\rout)$) -- (\x*\ang:\rout);
	
	\draw[color=\gcol] (\x*\ang:\rone) -- ($\geowtb*(\x*\ang:\rone)+\geowta*(\x*\ang+\geodev*\geodevmult:\rout)$);
	\draw[color=\gcol,rounded corners,decorate,decoration={random steps,segment length=\sele,amplitude=\amp}] ($\geowtb*(\x*\ang:\rone)+\geowta*(\x*\ang+\geodev*\geodevmult:\rout)$) -- (\x*\ang+\geodev:\rout);
}

\foreach \x/\n in {3/$x_1$}
{
	\draw[color=\gcol!\gcolcoeff!black,thick] (0,0) -- (\x*\ang:\rone) node[midway] {\tiny \textcolor{black}{{\n}}};		
	
	\draw[color=\gcol!\gcolcoeff!black,thick] (\x*\ang-\devthree:\rthree) -- (\x*\ang:\rone) node[near start] {\textcolor{black}{\tiny {$x_5$}}}; 
	\draw[color=\gcol!\gcolcoeff!black,thick] (\x*\ang:\rone) -- (\x*\ang+\devthree:\rthree) node[near end] {\tiny \textcolor{black}{{$x_3$}}};
	\draw[color=\gcol!\gcolcoeff!black,thick] (\x*\ang:\rone) -- (\x*\ang:\rthree) node[midway] {\tiny \textcolor{black}{{$x_4$}}};
	
	
	\draw[color=\gcol,rounded corners,decorate,decoration={random steps,segment length=\sele,amplitude=\amp}] (\x*\ang-\devthree:\rthree) -- (\x*\ang-\geodev:\rout);
	\draw[color=\gcol,rounded corners,decorate,decoration={random steps,segment length=\sele,amplitude=\amp}] (\x*\ang:\rthree) -- (\x*\ang:\rout);
	\draw[color=\gcol,rounded corners,decorate,decoration={random steps,segment length=\sele,amplitude=\amp}] (\x*\ang+\devthree:\rthree) -- (\x*\ang+\geodev:\rout);
}

\draw (0,0) circle (\rout);
\foreach \x/\n in {1/$x_5$,2/$x_6$,3/$x_1$,4/$x_2$,5/$x_3$,6/$x_4$}
{
	\draw[decorate,decoration={text along path,text align=center,text={block of {\n}}},color=white] (\x*\ang+\geodev:\rout+0.2) arc (\x*\ang+\geodev:\x*\ang-\geodev:\rout+0.2);
}

\draw[color=\acol,ultra thick] (4*\ang:\rout) arc (4*\ang:6*\ang-\geodev:\rout);

\foreach \x in {1,2,3,4,5,6}
{
	\filldraw (\x*\ang-\geodev:\rout) circle (2pt);
	\filldraw (\x*\ang:\rout) circle (2pt);
	\filldraw (\x*\ang+\geodev:\rout) circle (2pt);
}

\path (3*\ang-\geodev-2:\rout) node[below] {\tiny $\ray{x_1}{x_5}$};
\path (3*\ang:\rout) node[below] {\tiny $\ray{x_1}{x_4}$};
\path (3*\ang+\geodev+2:\rout) node[below] {\tiny $\ray{x_1}{x_3}$};

\end{scope}

\end{tikzpicture}}

	\caption{The nerve $C$, which is a 6-cycle, 
		part of the complex $\Sigma_C$ and the boundary $\bd{W_C}$; some geodesic rays, coming from 2-vertex special subgroups of the group $W_C$, and the blocks have been marked. We put arcs $\edge{x_2}{x_4}{x_3}{C}$ and $\edgeinf{x_4}{x_2}{x_2}{x_5}{x_3}{C}$ in bold.}
	\label{fig:cykwierzch}
\end{figure}
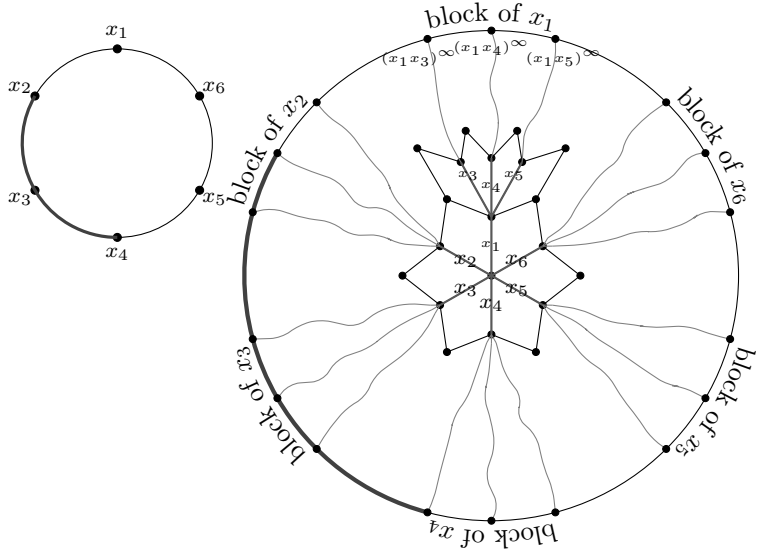

Consider a nerve $C$ homeomorphic to $S^1$, denote its 
consecutive
vertices as $x_1,\ldots,x_n$. 
For a fixed $i$, we call the set $\{\ray{x_i}{x_j}:x_i,x_j\mathrm{\;distinct,\;not\;adjacent\;in\;}C\}$ $\subseteq\bd{W_C}$ the \emph{block} of $x_i$.
In a cyclic order on $\bd{W_C}\cong S^1$, we have first the points of the block of $x_1$, then the points of the block of $x_2$, \ldots, the points of the block of $x_n$.  
Now, for vertices $x,x',y,y',z\in C^{(0)}$ such that $x,y,z$ are pairwise different, $x,x'$ distinct, non-adjacent, and $y,y'$ distinct, non-adjacent, denote by $\edgeinf{x}{x'}{y}{y'}{z}{C}$ the arc contained in the circle $\bd{W_C}$ that has endpoints $\ray{x}{x'},\ray{y}{y'}$ and contains any point (equivalently, all points) 
of the block of $z$. 
One may view it as a counterpart of the arc $\edge{x}{y}{z}{C}\subseteq C$ that has endpoints $x,y$ and contains vertex $z$.
See Figure \ref{fig:cykwierzch}.

The following lemma allows us to analyse the intersections of pairs of circles in terms of the intersections of the circles in the nerve that generate them.

\begin{lemma}\label{l:przeccyk}
Let $C_1,C_2$ be full subcomplexes of a nerve $N$ such that $C_1$ is a cycle and let 
$x,y$ be two different, non-adjacent 
vertices of $C_1$. 
Assume that there is a vertex $z\in C_1^{(0)}\setminus\{x,y\}$ such that 
we have
$\edge{x}{y}{z}{C_1}\cap C_2\subseteq\{x,y\}$. 
Then $\edgeinf{x}{y}{y}{x}{z}{C_1}\cap\bd{W_{C_2}}\subseteq\{\ray{x}{y},\ray{y}{x}\}$.
\end{lemma}

\begin{proof}
The key to the proof is the following observation on the set of 2-cubes of the complex $\Sigma_{C_1}$ that intersect the 
geodesic line
$\Sigma_{\{x,y\}}$:
on one side of the geodesic all such 2-cubes have edges labelled with vertices of the arc $\edge{x}{y}{z}{C_1}$, on the other side all such 2-cubes have edges labelled with vertices of the other arc in $C_1$ that has endpoints 
$x,y$.
Indeed, denote the vertices of $C_1$ in a cyclic order as $x,a_1,\ldots,a_l,y,z_1,\ldots,z_k$ (where $z\in\{z_1,\ldots,z_k\}$).
Consider the 
identity
vertex $e\in\Sigma_{C_1}$. 
The outgoing edges are labelled with $x,a_1,\ldots,a_l,y,z_1,\ldots,z_k$ (in a cyclic order). Consider the vertex $x\in\Sigma_{C_1}$. 
It is connected to the vertex $e$ by an edge labelled $x$, to which the 2-cube with edges labelled $x,a_1$ is attached, therefore the 
cyclic order of the edges incident to the vertex $x$ is opposite to the cyclic order of the edges incident to $e$.
This argument 
proves the observation locally. 
To finish, one may proceed by induction using a similar argument. 
 
Consider the geodesic ray $\varrho$ such that the point $[\varrho]$ is in
$\edgeinf{x}{y}{y}{x}{z}{C_1}\cap\bd{W_{C_2}}$.
By Proposition \ref{f:brzegpodkpl}, $\varrho$ is contained in the complex $\Sigma_{C_1}$.
Assume that $[\varrho]\neq\ray{x}{y},\ray{y}{x}$. 
Then there exists $t_0\geq0$ such that $\varrho([0,t_0])\subseteq\Sigma_{\{x,y\}}$ and $\varrho((t_0,\infty))\subseteq\Sigma_{C_1}\setminus\Sigma_{\{x,y\}}$.
By the key observation, 
since $[\varrho]\in\edgeinf{x}{y}{y}{x}{z}{C_1}$ and $\varrho\subseteq\Sigma_{C_1}\cap\Sigma_{C_2}$, 
it follows 
that $\varrho((t_0,\infty))$ is contained in both connected components of $\Sigma_{C_1}\setminus\Sigma_{\{x,y\}}$.
Contradiction.
\end{proof}

\begin{figure}
	\begin{tikzpicture}
\begin{scope}[xshift=-140,yshift=50]
\begin{scope}[scale=0.25,rotate=-90]
\newcommand{\rone}{5}
\newcommand{\ang}{60}

\draw (0,0) circle (\rone);
\foreach \x/\n in {1/$a_2$,2/$a_1$,3/$x$,4/$z_2$,5/$z_1$,6/$y$}
{
	\filldraw (\x*\ang:\rone) circle (6pt);
	\path (\x*\ang:\rone*1.18) node {\footnotesize {\n}};
}
\end{scope}
\end{scope}

\begin{scope}[rotate=-90,scale=0.65]
\newcommand{\ang}{60}
\newcommand{\scloc}{1.2}

\ifcolor
\newcommand{\colleft}{bro}
\newcommand{\colright}{gray}
\else
\newcommand{\colleft}{gray}
\newcommand{\colright}{darkgray}
\fi

\newcommand{\rout}{5}

\filldraw[fill=\colleft!30] (3*\ang:\rout) arc (3*\ang:6*\ang:\rout) -- cycle;

\newcommand{\bendval}{12}
\newcommand{\canlenone}{0.25*\ang};
\newcommand{\bendvaltwo}{0.5*\bendval}
\newcommand{\candivone}{0.26}
\newcommand{\candivmida}{0.33}
\newcommand{\candivtwo}{0.40}
\newcommand{\candivthree}{0.60}
\newcommand{\candivmidb}{0.67}
\newcommand{\candivfour}{0.74}
\newcommand{\canrtwoa}{3.3}\newcommand{\canrtwob}{3.5}

\newcommand{\rone}{1*\scloc}
\newcommand{\rtwo}{1.5*\scloc}
\newcommand{\offstwo}{\ang*0.5}
\newcommand{\rthree}{2.0*\scloc}
\newcommand{\devthree}{15}
\newcommand{\rfour}{2.5*\scloc}
\newcommand{\devfour}{0}
\newcommand{\devfourmid}{10}

\filldraw[color=\colright!45] 
(6*\ang:\rout) 
arc (6*\ang:7*\ang-\canlenone-\candivfour*\ang+\candivfour*\canlenone:\rout) 
to[bend left=\bendvaltwo] (\ang-\canlenone-\candivmidb*\ang+\candivmidb*\canlenone:\canrtwoa) 
to[bend left=\bendvaltwo] (\ang-\canlenone-\candivthree*\ang+\candivthree*\canlenone:\rout) 
arc (\ang-\canlenone-\candivthree*\ang+\candivthree*\canlenone:\ang-\canlenone-\candivtwo*\ang+\candivtwo*\canlenone:\rout) 
to[bend left=\bendvaltwo] (\ang-\canlenone-\candivmida*\ang+\candivmida*\canlenone:\canrtwoa) 
to[bend left=\bendvaltwo] (\ang-\canlenone-\candivone*\ang+\candivone*\canlenone:\rout)
arc (\ang-\canlenone-\candivone*\ang+\candivone*\canlenone:\ang-\canlenone:\rout) 
to[bend left=\bendval] (6*\ang+\offstwo:\rtwo)
-- (1*\ang:\rone) 
to[bend left=\bendval] (1*\ang:\rout)		
arc (1*\ang:\ang+\candivone*0.5*\ang-\candivone*0.5*\canlenone:\rout) 
to[bend left=\bendvaltwo] (\ang+\candivmida*0.5*\ang-\candivmida*0.5*\canlenone:\canrtwob) 
to[bend left=\bendvaltwo] (\ang+\candivtwo*0.5*\ang-\candivtwo*0.5*\canlenone:\rout) 
arc (\ang+\candivtwo*0.5*\ang-\candivtwo*0.5*\canlenone:\ang+\candivthree*0.5*\ang-\candivthree*0.5*\canlenone:\rout) 
to[bend left=\bendvaltwo] (\ang+\candivmidb*0.5*\ang-\candivmidb*0.5*\canlenone:\canrtwob) 
to[bend left=\bendvaltwo] (\ang+\candivfour*0.5*\ang-\candivfour*0.5*\canlenone:\rout)
arc (\ang+\candivfour*0.5*\ang-\candivfour*0.5*\canlenone:1*\ang+\offstwo-0.5*\canlenone:\rout) 
to[bend left=\bendval] (1*\ang+\offstwo:\rtwo) 
to[bend left=\bendval] (1*\ang+\offstwo+0.5*\canlenone:\rout)
arc (1*\ang+\offstwo+0.5*\canlenone:2*\ang-\candivfour*0.5*\ang+\candivfour*0.5*\canlenone:\rout) 
to[bend left=\bendvaltwo] (2*\ang-\candivmidb*0.5*\ang+\candivmidb*0.5*\canlenone:\canrtwob) 
to[bend left=\bendvaltwo] (2*\ang-\candivthree*0.5*\ang+\candivthree*0.5*\canlenone:\rout)
arc (2*\ang-\candivthree*0.5*\ang+\candivthree*0.5*\canlenone:2*\ang-\candivtwo*0.5*\ang+\candivtwo*0.5*\canlenone:\rout) 
to[bend left=\bendvaltwo] (2*\ang-\candivmida*0.5*\ang+\candivmida*0.5*\canlenone:\canrtwob) 
to[bend left=\bendvaltwo] (2*\ang-\candivone*0.5*\ang+\candivone*0.5*\canlenone:\rout)
arc (2*\ang-\candivone*0.5*\ang+\candivone*0.5*\canlenone:2*\ang:\rout) 
to[bend left=\bendval] (2*\ang:\rone)
-- (2*\ang+\offstwo:\rtwo) 
to[bend left=\bendval] (2*\ang+\canlenone:\rout)
arc (2*\ang+\canlenone:2*\ang+\canlenone+\candivone*\ang-\candivone*\canlenone:\rout) 
to[bend left=\bendvaltwo] (2*\ang+\offstwo+\devfour:\rfour)  
to[bend left=\bendvaltwo] (2*\ang+\canlenone+\candivtwo*\ang-\candivtwo*\canlenone:\rout)
arc (2*\ang+\canlenone+\candivtwo*\ang-\candivtwo*\canlenone:3*\ang-\ang+\canlenone+\candivthree*\ang-\candivthree*\canlenone:\rout) 
to[bend left=\bendvaltwo] (3*\ang-\devthree:\rthree) 
-- (3*\ang-\devfourmid:\rfour) 
to[bend left=\bendvaltwo] (3*\ang-\ang+\canlenone+\candivfour*\ang-\candivfour*\canlenone:\rout)
arc (3*\ang-\ang+\canlenone+\candivfour*\ang-\candivfour*\canlenone:3*\ang:\rout)
-- (6*\ang:\rout);

\filldraw (0,0) circle (2pt);

\foreach \x/\n in {1/$a_2$,2/$a_1$,3/$x$,4/$z_2$,5/$z_1$,6/$y$}
{
	\filldraw (\x*\ang:\rone) circle (2pt);
	\draw (0,0) -- (\x*\ang:\rone) node[midway] {\tiny {\n}};		
}

\foreach \x in {1,2,3,4,5,6}
{
	\filldraw (\x*\ang+\offstwo:\rtwo) circle (2pt);
	\draw (\x*\ang:\rone) -- (\x*\ang+\offstwo:\rtwo) -- (\x*\ang+\ang:\rone);
}

\foreach \x in {2}
{
	\draw (\x*\ang+\offstwo:\rtwo) -- (\x*\ang+\ang:\rone) node[midway] {\tiny {$a_1$}};
}

\foreach \x in {3}
{
	\draw (\x*\ang:\rone) -- (\x*\ang+\offstwo:\rtwo) node[midway] {\tiny {$z_2$}};
}

\foreach \x in {3}
{
	\filldraw (\x*\ang-\devthree:\rthree) circle (2pt);
	\filldraw (\x*\ang:\rthree) circle (2pt);
	\filldraw (\x*\ang+\devthree:\rthree) circle (2pt);
	\draw (\x*\ang:\rone) -- (\x*\ang-\devthree:\rthree) node[near end] {\tiny {$a_2$}};
	\draw (\x*\ang:\rone) -- (\x*\ang+\devthree:\rthree) node[near end] {\tiny {$z_1$}};
	\draw (\x*\ang:\rone) -- (\x*\ang:\rthree) node[midway] {\tiny {$y$}};
}

\foreach \x in {2}
{
	\filldraw (\x*\ang+\offstwo+\devfour:\rfour) circle (2pt);
	
	\draw (\x*\ang+\offstwo:\rtwo) -- (\x*\ang+\offstwo+\devfour:\rfour) -- (\x*\ang+\ang-\devthree:\rthree); 
}

\foreach \x in {3}
{
	\filldraw (\x*\ang+\offstwo-\devfour:\rfour) circle (2pt);
	
	\draw (\x*\ang+\devthree:\rthree) -- (\x*\ang+\offstwo-\devfour:\rfour) -- (\x*\ang+\offstwo:\rtwo); 
}

\foreach \x in {3}
{
	\filldraw (\x*\ang-\devfourmid:\rfour) circle (2pt);
	\filldraw (\x*\ang+\devfourmid:\rfour) circle (2pt);
	\draw (\x*\ang-\devthree:\rthree) -- (\x*\ang-\devfourmid:\rfour) -- (\x*\ang:\rthree) -- (\x*\ang+\devfourmid:\rfour) -- (\x*\ang+\devthree:\rthree);
}

\draw[thick] (0,0) circle (\rout);
\draw[very thick,color=\colleft!75!black] (3*\ang:\rout) arc (3*\ang:6*\ang:\rout);
\draw[thick] (3*\ang:\rout) -- (6*\ang:\rout); 
\filldraw (3*\ang:\rout) circle (2pt) node[above] {{\footnotesize $\ray{x}{y}$}};
\filldraw (6*\ang:\rout) circle (2pt) node[below] {\footnotesize {$\ray{y}{x}$}};


\foreach \x in {1}
{
	\draw (\x*\ang:\rone) to[bend left=\bendval] (\x*\ang:\rout); ;		
}

\foreach \x in {2}
{
	\draw (\x*\ang:\rout) to[bend left=\bendval] (\x*\ang:\rone);		
}

\foreach \x in {6}
{
	\draw (\x*\ang+\ang-\canlenone:\rout) to[bend left=\bendval] (\x*\ang+\offstwo:\rtwo);
}

\foreach \x in {1}
{
	\draw (\x*\ang+\offstwo-0.5*\canlenone:\rout) to[bend left=\bendval] (\x*\ang+\offstwo:\rtwo) to[bend left=\bendval] (\x*\ang+\offstwo+0.5*\canlenone:\rout);
}

\foreach \x in {2}
{
	\draw (\x*\ang+\offstwo:\rtwo) to[bend left=\bendval] (\x*\ang+\canlenone:\rout);
}

\foreach \x in {2}
{
	\draw (\x*\ang+\canlenone+\candivone*\ang-\candivone*\canlenone:\rout) to[bend left=\bendvaltwo] (\x*\ang+\offstwo+\devfour:\rfour)  to[bend left=\bendvaltwo] (\x*\ang+\canlenone+\candivtwo*\ang-\candivtwo*\canlenone:\rout); 
}

\foreach \x in {3}
{
	\draw (\x*\ang-\ang+\canlenone+\candivthree*\ang-\candivthree*\canlenone:\rout) to[bend left=\bendvaltwo] (\x*\ang-\devthree:\rthree) -- (\x*\ang-\devfourmid:\rfour) to[bend left=\bendvaltwo] (\x*\ang-\ang+\canlenone+\candivfour*\ang-\candivfour*\canlenone:\rout);	
}

\draw (\ang-\canlenone-\candivfour*\ang+\candivfour*\canlenone:\rout) to[bend left=\bendvaltwo] (\ang-\canlenone-\candivmidb*\ang+\candivmidb*\canlenone:\canrtwoa) to[bend left=\bendvaltwo] (\ang-\canlenone-\candivthree*\ang+\candivthree*\canlenone:\rout); 

\draw (\ang-\canlenone-\candivtwo*\ang+\candivtwo*\canlenone:\rout) to[bend left=\bendvaltwo] (\ang-\canlenone-\candivmida*\ang+\candivmida*\canlenone:\canrtwoa) to[bend left=\bendvaltwo] (\ang-\canlenone-\candivone*\ang+\candivone*\canlenone:\rout); 

\draw (\ang+\candivone*0.5*\ang-\candivone*0.5*\canlenone:\rout) to[bend left=\bendvaltwo] (\ang+\candivmida*0.5*\ang-\candivmida*0.5*\canlenone:\canrtwob) to[bend left=\bendvaltwo] (\ang+\candivtwo*0.5*\ang-\candivtwo*0.5*\canlenone:\rout);

\draw (\ang+\candivthree*0.5*\ang-\candivthree*0.5*\canlenone:\rout) to[bend left=\bendvaltwo] (\ang+\candivmidb*0.5*\ang-\candivmidb*0.5*\canlenone:\canrtwob) to[bend left=\bendvaltwo] (\ang+\candivfour*0.5*\ang-\candivfour*0.5*\canlenone:\rout);

\draw (2*\ang-\candivfour*0.5*\ang+\candivfour*0.5*\canlenone:\rout) to[bend left=\bendvaltwo] (2*\ang-\candivmidb*0.5*\ang+\candivmidb*0.5*\canlenone:\canrtwob) to[bend left=\bendvaltwo] (2*\ang-\candivthree*0.5*\ang+\candivthree*0.5*\canlenone:\rout);

\draw (2*\ang-\candivtwo*0.5*\ang+\candivtwo*0.5*\canlenone:\rout) to[bend left=\bendvaltwo] (2*\ang-\candivmida*0.5*\ang+\candivmida*0.5*\canlenone:\canrtwob) to[bend left=\bendvaltwo] (2*\ang-\candivone*0.5*\ang+\candivone*0.5*\canlenone:\rout);

\end{scope}

\end{tikzpicture}
	\centering
	\caption{The situation in the proof of Lemma \ref{l:przeccyk} in case when the nerve $C_1$ is a 6-cycle 
	and $(C_1\cap C_2)^{(0)}=\{x,a_1,a_2,y\}$ (one may show then that $\Sigma_{C_1\cap C_2}(=\Sigma_{C_1}\cap\Sigma_{C_2})$ resembles the Cantor tree and $\bd{W_{C_1\cap C_2}}(=\bd{W_{C_1}}\cap\bd{W_{C_2}})$ is topologically the Cantor set). 
	In lighter grey we marked the fragment of the complex $\Sigma_{C_1}$, through which the geodesics giving the arc $\edgeinf{x}{y}{y}{x}{z_1}{C_1}$ go and this arc in the boundary $\bd{W_{C_1}}$, in darker grey we marked the complex $\Sigma_{C_1\cap C_2}$.}
	\label{fig:przeccyk}
\end{figure}
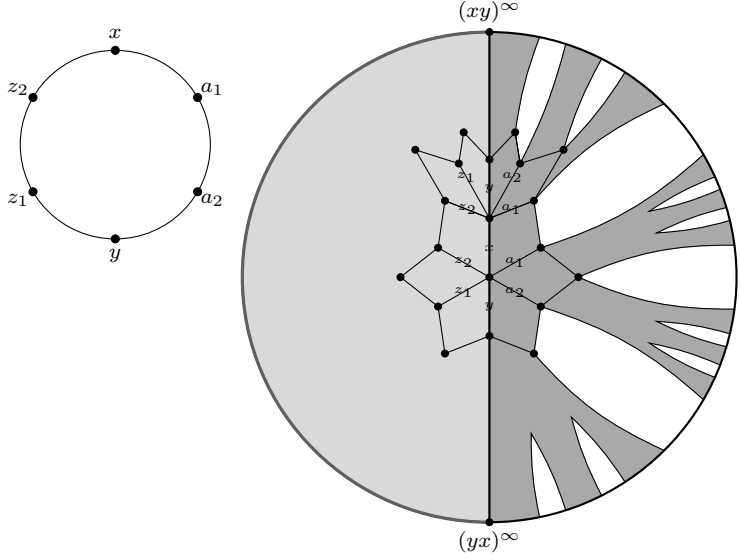

\begin{figure}
	\iflight
	\begin{tikzpicture}
\newcommand{\rone}{5}
\newcommand{\ang}{72}
\ifcolor
\newcommand{\colone}{red}
\newcommand{\coltwo}{blue}
\newcommand{\colthree}{darkgreen}
\newcommand{\colfour}{afb}
\newcommand{\colfive}{bro}
\else
\newcommand{\colone}{black!15}
\newcommand{\coltwo}{black!15}
\newcommand{\colthree}{black!50}
\newcommand{\colfour}{black!15}
\newcommand{\colfive}{black!15}
\fi

\begin{scope}[xshift=-150,yshift=-50]
\begin{scope}[rotate=90,scale=0.29]
\draw (0,0) circle (\rone);

\foreach \x/\c in {1/\colone,2/\coltwo,3/\colthree,4/\colfour,5/\colfive}
{
	\draw[color=\c] (\x*\ang:\rone) -- (\x*\ang+2*\ang:\rone);
}

\foreach \x in {1,2,3,4,5}
{
	\filldraw (\x*\ang:\rone) circle (4pt);
}

\foreach \x in {1,3,4}
{
	\path (\x*\ang-\ang:\rone*1.17) node {\footnotesize {$x_{\x}$}};
} 

\foreach \x in {1,2,3}
{
	\filldraw (0.25*\ang*\x:\rone) circle (2pt);
}

\foreach \x in {1,2}
{
	\filldraw (2*\ang+0.33*\ang*\x:\rone) circle (2pt);
}

\filldraw[color=\colone] ($0.5*(\ang:\rone)+0.5*(3*\ang:\rone)$) circle (2pt); 

\filldraw[color=\coltwo] ($0.75*(2*\ang:\rone)+0.25*(4*\ang:\rone)$) circle (2pt); 
\filldraw[color=\coltwo] ($0.5*(2*\ang:\rone)+0.5*(4*\ang:\rone)$) circle (2pt); 
\filldraw[color=\coltwo] ($0.25*(2*\ang:\rone)+0.75*(4*\ang:\rone)$) circle (2pt); 

\filldraw[color=\colthree] ($0.33*(3*\ang:\rone)+0.67*(5*\ang:\rone)$) circle (2pt); 
\filldraw[color=\colthree] ($0.67*(3*\ang:\rone)+0.33*(5*\ang:\rone)$) circle (2pt); 

\filldraw[color=\colfour] ($0.5*(4*\ang:\rone)+0.5*(6*\ang:\rone)$) circle (2pt);

\path (55:1.4*\rone) node {\Large {$\Gamma$}};
\end{scope}
\end{scope}

\begin{scope}[xshift=-150,yshift=50]
\begin{scope}[rotate=90,scale=0.22]
\draw (0,0) circle (\rone);

\foreach \x/\c in {1/\colone,2/\coltwo,3/\colthree,4/\colfour,5/\colfive}
{
	\draw[color=\c] (\x*\ang:\rone) -- (\x*\ang+2*\ang:\rone);
}

\foreach \x in {1,2,3,4,5}
{
	\filldraw (\x*\ang:\rone) circle (4pt);
}

\path (55:1.4*\rone) node {\Large {$G$}};

\end{scope}
\end{scope}

\begin{scope}
\begin{scope}[rotate=90,scale=0.5]
\newcommand{\dev}{15}

\foreach \x in {1,2,3,4,5}
{
	\draw[dashed] (\x*\ang-\dev:\rone) arc (\x*\ang-\dev:\x*\ang+\dev:\rone);
	\draw (\x*\ang+\dev:\rone) arc (\x*\ang+\dev:\x*\ang+\ang-\dev:\rone);	
}  

\foreach \x/\c in {1/\colone,2/\coltwo,3/\colthree,4/\colfour,5/\colfive}
{
	\draw[color=\c] (\x*\ang+\dev:\rone) to[bend left] (\x*\ang+2*\ang-\dev:\rone);
}

\foreach \x in {1,2,3,4,5}
{
	\filldraw (\x*\ang-\dev:\rone) circle (2pt);
	\filldraw (\x*\ang+\dev:\rone) circle (2pt);
}

\path (\dev:\rone*1.1) node {\footnotesize {$\ray{x_1}{x_3}$}};
\path (-\dev:\rone*1.1) node {\footnotesize {$\ray{x_1}{x_4}$}};
\path (2*\ang-\dev:\rone*1.13) node {\footnotesize {$\ray{x_3}{x_1}$}};
\path (3*\ang+\dev:\rone*1.15) node {\footnotesize {$\ray{x_4}{x_1}$}};

\path (-5:1.24*\rone) node[right] {\Large {$H\subseteq\bd{W_N}$}};
\end{scope}
\end{scope}

\end{tikzpicture}
	\else
	\includegraphics[width=0.9\textwidth]{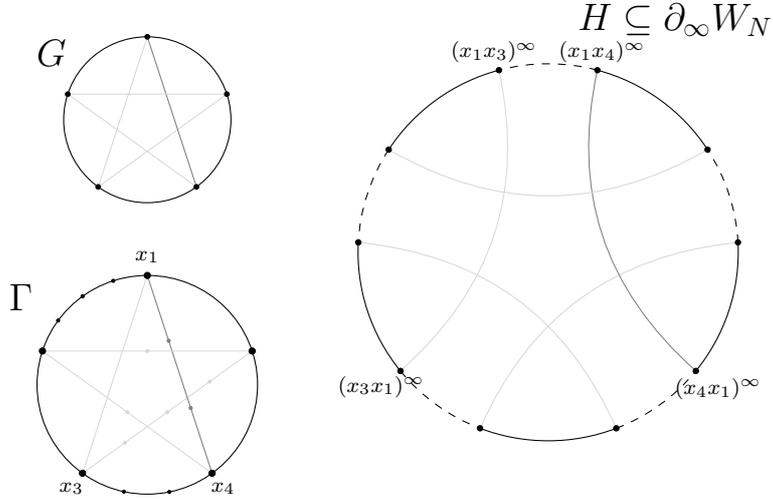}
	\fi
	\centering
	\caption{\emph{Left:} Situation in the assumption of Theorem \ref{t:nieplan}: subcomplex $\Gamma$, which is an edge subdivision of some graph $G\cong\kapi$. \emph{Right:} The embedding of the graph $H$, which is constructed in the proof. We marked the cycle $C$ (in $\Gamma$), the cycle $D$ (in $G$) and the edges of type Ia (in $H$) with continuous black lines, and the edges of type Ib (in $H$) with dashed lines. In grey we marked the arcs $S_i$ (in $\Gamma$), the edges corresponding to them in $G$ and the edges of type II (in $H$).
	In dark grey we marked one edge of $G$, and a path in $\Gamma$ and an arc in $H$ that correspond to this edge.
}
	\label{fig:tnieplan}
\end{figure}

\subsection{Construction}\label{s:nieplandowod}

Let us introduce two definitions.

\begin{definition}
\begin{enumerate}[(i)]
\item A graph is a \emph{weak minor} of the graph $G$, if it can be obtained from the graph $G$ by a sequence of edge contractions. 

\item A graph is an \emph{edge subdivision} of the graph $G$, if it can be obtained from $G$ by subdividing each of its edges (we 
use the
convention that a trivial subdivision is a subdivision).   
\end{enumerate}
\end{definition}

\begin{remark}
If some non-planar graph is a weak minor of the graph $G$, then $G$ itself is non-planar.
\end{remark}

Recall that a graph is \emph{simple} if it has no loops and no multiple edges; \emph{Hamiltonian} if it contains a \emph{Hamiltonian cycle}, that is, a cycle that passes 
through each 
vertex exactly once.

\begin{theorem}\label{t:nieplan}
Let $\Gamma$ be an edge subdivision of a Hamiltonian simple graph $G$ whose all vertices have degree at least $3$. Assume that $\Gamma$ decomposes into a cycle $C$, which is an edge subdivision of some Hamiltonian cycle $D$ in $G$, and a collection of paths $S_1,\ldots,S_k$ that intersect $C$ only at their both endpoints. Suppose that $N$ is a nerve whose $1$-skeleton contains $\Gamma$, that $C$ is a full subcomplex of $N$ (in particular, each $S_i$ has at least two edges), and that there exists a collection of arcs $L_1,\ldots,L_k$ contained in $C$ such that for each $i$ the arc $L_i$ has the same endpoints as $S_i$ and $S_i \cup L_i$ is a full subcomplex of $N$. Then some graph $H$ such that $G$ is a weak minor of $H$ embeds into the boundary $\partial_\infty W_N$.
\end{theorem}

\begin{remark}\label{u:tnieplan}
If the graph $\Gamma$ is a full subcomplex of $N$ and we choose a cycle $C$, then the condition that $S_i\cup L_i$ is a full subcomplex of $N$ becomes trivial, so the choice of the $L_i$ does not matter. 
Additionally, if we assume that 
for
each edge of $G$ there 
is a corresponding
path of length at least $2$ in $\Gamma$, then the choice of the cycle $C$ does not matter. 
\end{remark}

\begin{definition}\label{d:nieplan}
A flag simplicial complex $N$ satisfying the assumption of Theorem \ref{t:nieplan} for some non-planar graph $G$ is called \emph{\sg-non-planar}.
\end{definition}

\begin{proof} ({\sc of Theorem \ref{t:nieplan}})
See Figure \ref{fig:tnieplan}. 
For simplicity, we do not distinguish a graph $H$ 
as in the 
claim
from its embedding in the boundary $\bd{W_N}$ and we identify the vertices of $G$ with the corresponding 
vertices of the graph $\Gamma$. 

\smallskip
\emph{Construction.} Let $x_1,\ldots,x_n$ be the vertices of the cycle $D$.
Define the vertex set of the graph $H$
to be 
$V_H=\{\ray{a_1}{b_1},\ray{b_1}{a_1},\ldots,\ray{a_k}{b_k},\ray{b_k}{a_k}\}$, where $a_i,b_i$ are the endpoints of $S_i$.
Notice that the rays are well defined since $C$ is a full subcomplex of $N$ and $G$ is simple. 
Clearly $V_H\subseteq\bd{W_C}$. 
The edges will come from two sources, we split them into three types.
The first source is $\bd{W_C}$ -- we take the subdivision of $\bd{W_C}$ induced by $V_H$. We obtain $2$ types of edges: for each $i$ we have the edge between a vertex from the block $x_i$ and a vertex from the block of $x_{i+1}$ 
(type Ia), and the edges between vertices from the 
block $x_i$ 
(type Ib). 
The second source are the subcomplexes $S_i\cup L_i$. For each path $S_i$ having endpoints $x,y$, fix any vertex $z\in S_i$ other than $x,y$ and add to $H$ an edge $\edgeinf{x}{y}{y}{x}{z}{S_i\cup L_i}$ (type II).

\smallskip
\emph{Correctness.}
Consider a pair $e_1,e_2$ of different edges. We will show that their intersection is either empty or consists of a common endpoint of $e_1$ and $e_2$. If both $e_1,e_2$ are of the types Ia or Ib, the claim is clear. If exactly one of them is of type II, the claim follows easily by Lemma \ref{l:przeccyk}.
Assume that $e_1,e_2$ correspond to two different paths $S_i,S_j$ (with endpoints $a_i,b_i,a_j,b_j$ respectively). Then, by Lemma \ref{l:przeccyk}, we have that $e_1\cap e_2\subseteq (e_1\cap\bd{W_{S_j\cup L_j}})\cap(\bd{W_{S_i\cup L_i}}\cap e_2)\subseteq\{\ray{a_i}{b_i},\ray{b_i}{a_i}\}\cap\{\ray{a_j}{b_j},\ray{b_j}{a_j}\}=\emptyset$. The last equality holds since $\{a_i,b_i\}\neq\{a_j,b_j\}$.

The graph $H$ is obtained from the graph $G$ in the following way.
We start with the cycle $D\subseteq G$. For each $i$ we 
replace
the vertex $x_i$ with a path $B_i$ having $\deg_Gx_i-2$ vertices.
Then the original edges of the cycle $D$ correspond to the edges of type Ia and the paths $B_i$ are realised by the edges of type Ib. 
Furthermore, 
for
each edge $\{x_i,x_j\}$ of the graph $G$ not belonging to the cycle $D$ 
there 
is a corresponding
edge between 
a vertex in $B_i$ and a vertex in $B_j$.
It follows that the graph $G$ is obtained from the graph $H$ by contracting all the edges of type Ib.
\end{proof}

\subsection{Proof of the main theorem}\label{c:dowod}

\begin{proof}{(\sc of Theorem \ref{t:main})}
We check that the boundary $\bd{W_N}$ satisfies the conditions from Proposition \ref{f:charmenger}.	

\smallskip
\emph{Metrisability, compactness.} See Remark \ref{u:wlbrzeg}\ref{u:wlbrzegzwme}.	
	
\smallskip
\emph{1-dimensionality.} We have the following equality: $\dim\bd{W_N}\!=\!\max\{n\!:\!\widetilde{H}^n(N)\neq0\;\mathrm{or}\;\widetilde{H}^n(N\setminus\Delta)\neq0\;\mathrm{for\;some\;simplex}\,\Delta\!\subseteq\! N\}$ (where $\widetilde{H}^*$ denotes the reduced cohomology), \cite[proof of Lemma 2.5]{Swiat16}.
By the assumption, $\dim\bd{W_N}\leq1$. 
If we had $\dim\bd{W_N}=0$, then $W_N$ would be virtually free,  \cite[Corollary 8.5.6]{DavBook08}. On the other hand, since $N$ is connected, 
not
a simplex and has no separating simplex, $W_N$ is 1-ended, \cite[Theorem 8.7.2]{DavBook08}, a contradiction.

\smallskip
\emph{Connectedness.} 
Since $W_N$ is $\catz$ and 1-ended, $\bd{W_N}$ is connected.

\smallskip
\emph{Local connectedness.} 
Since $W_N$ is 1-ended and hyperbolic, the boundary $\bd{W_N}$ has no global cut-points \cite{Swarup96,Bowditch99}, thus by \cite{BesMes91} it is locally connected.

\smallskip
\emph{No local cut-points.} 
By inseparability of $N$ and \cite{MicTsc09} the group $W_N$ does not split over a 2-ended or a finite subgroup.
By \cite{Bowditch98} the boundary $\bd{W_N}$ has no local cut-points or is a cocompact Fuchsian group.
In the latter case, by \cite[Theorem 10.9.2]{DavBook08}, $N$ is either a triangulation of $S^1$ or a join of a triangulation of $S^1$ with a simplex. Both of these cases contradict inseparability of the nerve $N$.

\smallskip
\emph{No planar open subsets.} By Theorem \ref{t:nieplan} there is an embedding of some non-planar graph $H$ into the boundary $\bd{W_N}$. By \cite[Lemma 7]{KaKl98}, if there were some planar neighbourhood of some point in the boundary $\bd{W_N}$, then the graph $H$ would embed into this neighbourhood.
\end{proof}

\begin{remark}\label{u:mainkonieczne}
\begin{enumerate}[(i)]
	\item By the formula for $\dim\bd{W_N}$ given in the above proof, the condition on cohomology in the assumption of Theorem \ref{t:main} is a necessary condition for the boundary $\bd{W_N}$ to be 1-dimensional. 
	\label{u:mainkonieczne1}	
	
	\item The condition that $N$ is inseparable and is not a simplex is necessary for the boundary $\bd{W_N}$ to be connected and have no local cut-points, \cite[Lemma 2.2]{Swiat16}.
	\label{u:mainkonieczne2}
	
	\item The observations from above two remarks and the proofs of compactness, metrisability, 1-dimensionality and connectedness do not require the assumption that the group $W_N$ is hyperbolic. 
	Omitting or weakening of the assumption of hyperbolicity in the other three parts of the proof seems to be a non-trivial task. 
	
	\item The boundary of a Coxeter group having a planar nerve is planar (this does not require hyperbolicity), \cite[Lemma 2.4]{Swiat16}, therefore non-planarity of the nerve $N$ is a necessary condition for the boundary $\bd{W_N}$ to be the Menger curve. 
	\label{u:mainkonieczne4}

\end{enumerate}
\end{remark}

\section{Applications}\label{c:prz}
In this section we use Theorem \ref{t:main} to find triangulations of some topological spaces, that give, as nerves, right-angled Coxeter groups with Menger curve boundary.

\begin{remark}
From now on, whenever we consider the boundary $\partial\sigma$ of a simplicial complex $\sigma$, the complex $\sigma$ is a triangulation of a manifold $M_\sigma$ with boundary, and by its boundary $\partial\sigma$ we mean the subcomplex of $\sigma$ that corresponds to the boundary $\partial M_\sigma$.   
In particular, if $\sigma$ is a subcomplex of some simplicial complex $\tau$, we do not mean the topological boundary of $\sigma$ in the space $\tau$.
\end{remark}

\subsection{Triangulations of 2-dimensional simplicial complexes}

\begin{definition}
	An edge of a 2-dimensional simplicial complex $X$ that is not contained in any face of $X$ is a  \emph{lonely edge}.
	\label{d:lonelyedge}   
\end{definition}

\begin{theorem}\label{t:kompleksy}
	Let $X$ be a 
	2-dimensional simplicial complex without lonely edges.
	Then the following are equivalent:
	\begin{enumerate}[(i)]
		\item $X$ admits a triangulation $N$ such that $N$ is flag and the boundary $\bd{W_N}$ is the Menger curve,
		\label{t:kompleksy1}
		
		\item $X$ admits infinitely many triangulations as in \ref{t:kompleksy1},
		\label{t:kompleksy2}
		
		\item $X$ is connected, non-planar, has no separating pair of points, and $H^2X=0$.
		\label{t:kompleksy3}
	\end{enumerate}
\end{theorem}

\begin{remark}
\begin{enumerate}[(i)]
\item Having no lonely edges implies that a point $x\in X$ is a local cut-point 
if and only if it is a vertex of $X$ whose link is not connected.
Therefore, the set of local cut-points in $X$ is discrete, so any triangulation $N$ of $X$ has no lonely edges. It follows that each separating point of $X$ is a vertex of $N$.
\item If $X$ has no separating pair, then it has no separating point.
\end{enumerate}
\end{remark}

\begin{corollary}\label{c:powierzchnie}
Let $M$ be a compact surface (possibly with boundary). Then $M$ admits a triangulation (equivalently, infinitely many triangulations) that is the nerve of a right-angled Coxeter group with Menger curve boundary if and only if the boundary $\partial M$ is non-empty and $M$ is non-planar.  	
\end{corollary}
\begin{proof}{\sc (of Corollary \ref{c:powierzchnie})}
We can triangulate $M$. Such a triangulation has no local cut-points and no lonely edges. In view of Theorem \ref{t:kompleksy}, it suffices to check that $\partial M\neq\emptyset$ if and only if $H^2M=0$, but this follows from classical theorems on surfaces and manifolds. 
\end{proof}

\begin{proof}{\sc (of Theorem \ref{t:kompleksy})}
\emph{\ref{t:kompleksy2}$\Rightarrow$\ref{t:kompleksy1}.}
This implication is obvious.

\medskip	
\emph{\ref{t:kompleksy1}$\Rightarrow$\ref{t:kompleksy3}.} Non-planarity of $X$ follows by Remark \ref{u:mainkonieczne}\ref{u:mainkonieczne4}. 
The condition $H^2X=0$ follows by Remark \ref{u:mainkonieczne}\ref{u:mainkonieczne1}.
Connectedness of $X$ and the lack of separating pair of points in $X$ are necessary conditions for the existence of an inseparable triangulation of $X$, which by Remark \ref{u:mainkonieczne}\ref{u:mainkonieczne2} finishes the proof of this implication.
	
\medskip	
\emph{\ref{t:kompleksy3}$\Rightarrow$\ref{t:kompleksy2}.}
In the proof, first we find a non-planar graph that is embedded in $X$, 
and then 
we make it a part of some triangulation, 
which we use to produce
produce infinitely many ones that satisfy the assumptions of Theorem \ref{t:main}.

\smallskip
\emph{Finding the graph.} 
The space $X$ is compact, metrisable, 
connected, locally connected, non-planar and has no cut-points, therefore by \cite{Claytor34} a graph isomorphic to $\katt$ or $\kapi$ 
embeds in $X$. 
(Notice that we do not need to use such a general theorem in the case when $X$ is a less general space, e.g. a surface).
We may assume that $X$ admits a triangulation $K$ that has a 1-dimensional subcomplex $\Gamma$ which is an edge subdivision of either $\katt$ or $\kapi$.

\begin{figure}
\iflight
	\centerline{\begin{tikzpicture}
\newcommand{\capx}{2}
\newcommand{\capy}{6.7}

\newcommand{\add}{0.31}
\newcommand{\base}{(1-\add)*0.3333}

\begin{scope}[scale=0.8]

\newcommand{\dran}[3]{
\coordinate (A) at (#1);
\coordinate (B) at (#2);
\coordinate (C) at (#3);

\coordinate (D) at ($0.5*(B)+0.5*(C)$);
\coordinate (E) at ($0.5*(A)+0.5*(C)$);
\coordinate (F) at ($0.5*(A)+0.5*(B)$);

\coordinate (G) at ($\base*(A)+\base*(B)+\base*(C)+\add*(A)$);
\coordinate (H) at ($\base*(A)+\base*(B)+\base*(C)+\add*(B)$);
\coordinate (I) at ($\base*(A)+\base*(B)+\base*(C)+\add*(C)$);

\draw (A) -- (B) -- (C) -- (A) -- (G) -- (F) -- (H) -- (D) -- (I) -- (E) -- (G) -- (H) -- (B) -- (H) -- (I) -- (C) -- (I) -- (G);}

\coordinate (a) at (-5,0);
\coordinate (b) at (5,0);
\coordinate (c) at (0,8);

\coordinate (d) at ($0.5*(b)+0.5*(c)$);
\coordinate (e) at ($0.5*(a)+0.5*(c)$);
\coordinate (f) at ($0.5*(a)+0.5*(b)$);

\coordinate (g) at ($\base*(a)+\base*(b)+\base*(c)+\add*(a)$);
\coordinate (h) at ($\base*(a)+\base*(b)+\base*(c)+\add*(b)$);
\coordinate (i) at ($\base*(a)+\base*(b)+\base*(c)+\add*(c)$);

\filldraw (a) circle (2pt);
\filldraw (b) circle (2pt);
\filldraw (c) circle (2pt);
\filldraw (d) circle (2pt);
\filldraw (e) circle (2pt);
\filldraw (f) circle (2pt);
\filldraw (g) circle (2pt);
\filldraw (h) circle (2pt);
\filldraw (i) circle (2pt);

\draw[semithick] (a) -- (b) -- (c) -- (a);
\draw[thick] (a) -- (g) -- (f) -- (h) -- (d) -- (i) -- (e) -- (g) -- (h) -- (b) -- (h) -- (i) -- (c) -- (i) -- (g);

\dran{a}{g}{e};
\dran{a}{g}{f};
\dran{b}{h}{f};
\dran{b}{h}{d};
\dran{c}{i}{d};
\dran{c}{i}{e};

\dran{g}{h}{f};
\dran{h}{i}{d};
\dran{i}{g}{e};

\dran{g}{h}{i};

\path (\capx,\capy) node {\huge {$\Delta^{\dra\dra}$}};
\end{scope}

\begin{scope}[xshift=-130,yshift=85]
\begin{scope}[scale=0.4]
\coordinate (a) at (-5,0);
\coordinate (b) at (5,0);
\coordinate (c) at (0,8);

\coordinate (d) at ($0.5*(b)+0.5*(c)$);
\coordinate (e) at ($0.5*(a)+0.5*(c)$);
\coordinate (f) at ($0.5*(a)+0.5*(b)$);

\coordinate (g) at ($\base*(a)+\base*(b)+\base*(c)+\add*(a)$);
\coordinate (h) at ($\base*(a)+\base*(b)+\base*(c)+\add*(b)$);
\coordinate (i) at ($\base*(a)+\base*(b)+\base*(c)+\add*(c)$);

\filldraw (a) circle (2pt);
\filldraw (b) circle (2pt);
\filldraw (c) circle (2pt);
\filldraw (d) circle (2pt);
\filldraw (e) circle (2pt);
\filldraw (f) circle (2pt);
\filldraw (g) circle (2pt);
\filldraw (h) circle (2pt);
\filldraw (i) circle (2pt);

\draw[semithick] (a) -- (b) -- (c) -- (a);
\draw[thick] (a) -- (g) -- (f) -- (h) -- (d) -- (i) -- (e) -- (g) -- (h) -- (b) -- (h) -- (i) -- (c) -- (i) -- (g);

\path (-\capx,\capy) node {\LARGE {$\Delta^{\dra}$}};

\end{scope}
\end{scope}

\begin{scope}[xshift=-160,yshift=10]
\begin{scope}[scale=0.3]
\coordinate (a) at (-5,0);
\coordinate (b) at (5,0);
\coordinate (c) at (0,8);
\draw (a) -- (b) -- (c) -- (a);

\path (-\capx,\capy) node {\Large {$\Delta$}};

\end{scope}
\end{scope}

\end{tikzpicture}}
\else
	\includegraphics[width=0.9\textwidth]{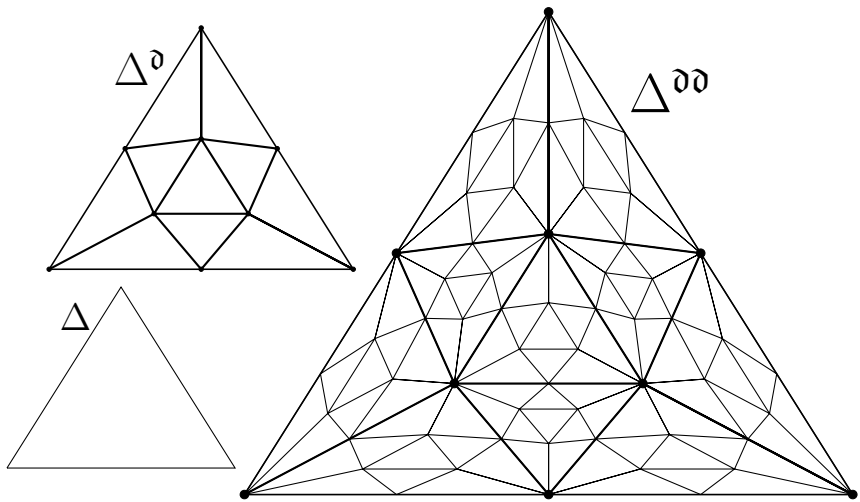}
\fi
\centering
\caption{The simplex $\Delta$ with its subdivisions $\Delta^{\dra}$ and $\Delta^{\dra\dra}$.}
\label{fig:podrozbiciedra}
\end{figure}

\smallskip
\emph{Final triangulation.} 
We will modify $K$ in such a way that it satisfies assumptions of Theorem \ref{t:main}.
We use a method of subdividing 2-dimensional complexes that was introduced in \cite{Dra99}. It consists in subdividing each edge into two edges and subdividing each face as in Figure \ref{fig:podrozbiciedra}. We denote by $L^{\dra}$ the result of such subdivision applied to a simplicial complex $L$. 
It has the following properties.

\begin{lemma}\label{l:podrozbiciedra}
Let $L$ be a 2-dimensional simplicial complex. Then: 
\begin{enumerate}[(i)]
\item $L^{\dra}$ is flag no-$\square$,
\label{l:podrozbiciedra1}

\item if $\Gamma$ is a 1-dimensional subcomplex of $L$, then the subcomplex $\Gamma^{\dra}$ of $L^{\dra}$
is full,
\label{l:podrozbiciedra2}

\item if $L$ is connected, has no lonely edges an no separating pair of vertices, then the complex $L^{\dra\dra}$ is inseparable. 
\label{l:podrozbiciedra3} 
\end{enumerate}	
\end{lemma}    

We omit the proof of Lemma \ref{l:podrozbiciedra} since it is a combination of well known properties and some routine combinatorial reasonings.
As the required infinitely many triangulations, we take $n$-fold subdivisions $K^{n\times\dra}$ for $n\geq2$ --
by Lemma \ref{l:podrozbiciedra} each of these triangulations satisfies all the assumptions of Theorem \ref{t:main} 
apart from
the cohomology condition (for \sg-non-planarity, recall Remark \ref{u:tnieplan}). The latter can be shown by the following standard reasoning using the Mayer-Vietoris sequence. 
Let $\Delta$ be any simplex of $K^{n\times\dra}$.
By the properties of simplicial complexes, there exists a neighbourhood $U$ of $\Delta$ such that the boundary $\partial U$ is homeomorphic to a 1-dimensional simplicial complex, $\Delta$ is a deformation retract of $U$, and $U\setminus\Delta$ is homotopy equivalent to $\partial U$. 
Then, considering the following part of the Mayer-Vietoris sequence: $H^2(K^{n\times\dra})\to H^2(K^{n\times\dra}\setminus\Delta)\oplus H^2(U)\to H^2(U\cap(K^{n\times\dra}\setminus\Delta))$ whose left term is 0 by assumption and the right term is 0 by the choice of $U$, we have 
$H^2(K^{n\times\dra}\setminus\Delta)=0$.
\end{proof}

\subsection{Triangulations of disks $D^n$}

In this section we prove the following theorem.

\begin{theorem}\label{t:dysk}
	The disk $D^n$ admits a triangulation that is a nerve of a right-angled Coxeter group with Menger curve boundary if and only if $n\geq3$.
\end{theorem}

Note that, by the cohomology condition in Theorem \ref{t:main}, 
the desired triangulation should have all its $(n-3)$-simplicies contained in the boundary $\partial D^n$.

\begin{proof}
The 
essence
of the proof is the case $n=3$. To cover this case, we construct a triangulation of $D^3$ that satisfies the conditions of Theorem \ref{t:main}. When $n<3$, the disk $D^n$ is planar, so 
any flag triangulation of $D^n$
(viewed as a nerve) yields a right-angled Coxeter group with planar
boundary, \cite[Lemma 2.4]{Swiat16}.
On the other hand, if $N$ is a flag no-$\square$ complex, then 
so is the simplicial cone $\Cone{N}$ over $N$,
and we have $W_{\Cone{N}}\cong W_N\oplus\Z_2$.
This implies that $\bd{W_N}\cong\bd{W_{\Cone{N}}}$. Since $\Cone{D^n}\cong D^{n+1}$, in order to obtain an appropriate triangulation of $D^n$ for $n>3$ it suffices to take $n-3$ times the simplicial cone over the triangulation for $D^3$. From now on we concentrate on the case 
$n=3$.

\begin{figure}
\centering
	\centerline{\begin{tikzpicture}
\begin{scope}[xshift=-180,yshift=40]
\begin{scope}[scale=0.4]
\newcommand{\h}{5}

\newcommand{\w}{2}
\foreach \x in {-4,-3,-2,-1,0,1,2,3,4}
{
	\draw[color=\cone,thick] (0,\h) -- (\x*\w,0) -- (0,-\h);
}
\draw[color=\cone,thick] (-4*\w,0) -- (4*\w,0);

\newcommand{\htwoy}{0.5}
\newcommand{\htwox}{0.5*(1-\htwoy)}
\foreach \x in {-4,-3,-2,-1,0,1,2,3}
{
	
	\draw[color=\ctwo] ($\htwox*(\x*\w,0)+\htwox*(\x*\w+\w,0)+\htwoy*(0,\h)$) -- (\x*\w,0);
	\draw[color=\ctwo] ($\htwox*(\x*\w,0)+\htwox*(\x*\w+\w,0)+\htwoy*(0,\h)$) -- (\x*\w+\w,0);
	\draw[color=\ctwo] ($\htwox*(\x*\w,0)+\htwox*(\x*\w+\w,0)+\htwoy*(0,\h)$) -- (0,\h);
	
	\draw[color=\ctwo] ($\htwox*(\x*\w,0)+\htwox*(\x*\w+\w,0)+\htwoy*(0,-\h)$) -- (\x*\w,0);
	\draw[color=\ctwo] ($\htwox*(\x*\w,0)+\htwox*(\x*\w+\w,0)+\htwoy*(0,-\h)$) -- (\x*\w+\w,0);
	\draw[color=\ctwo] ($\htwox*(\x*\w,0)+\htwox*(\x*\w+\w,0)+\htwoy*(0,-\h)$) -- (0,-\h);
}

\foreach \x in {-4,-3,-2,-1,0,1,2}
{
	\draw[color=\cthree] ($\htwox*(\x*\w,0)+\htwox*(\x*\w+\w,0)+\htwoy*(0,\h)$) -- ($\htwox*(\x*\w+\w,0)+\htwox*(\x*\w+2*\w,0)+\htwoy*(0,\h)$);
	\draw[color=\cthree] ($\htwox*(\x*\w,0)+\htwox*(\x*\w+\w,0)+\htwoy*(0,-\h)$) -- ($\htwox*(\x*\w+\w,0)+\htwox*(\x*\w+2*\w,0)+\htwoy*(0,-\h)$);
}

\filldraw[color=\cone] (0,\h) circle (2pt);
\filldraw[color=\cone] (0,-\h) circle (2pt);

\foreach \x in {-4,-3,-2,-1,0,1,2,3,4}
{
	\filldraw[color=\cone,thick] (\x*\w,0) circle (2pt);
}

\foreach \x in {-4,-3,-2,-1,0,1,2,3}
{
	\filldraw[color=\ctwo] ($\htwox*(\x*\w,0)+\htwox*(\x*\w+\w,0)+\htwoy*(0,\h)$) circle (2pt);
	\filldraw[color=\ctwo] ($\htwox*(\x*\w,0)+\htwox*(\x*\w+\w,0)+\htwoy*(0,-\h)$) circle (2pt);
}


\end{scope}
\end{scope}

\begin{scope}[scale=0.6]
\newcommand{\h}{5}

\newcommand{\w}{2}
\foreach \x in {-4,-3,-2,-1,0,1,2,3,4}
{
	\draw[color=\cone,thick] (0,\h) -- (\x*\w,0) -- (0,-\h);
}
\draw[color=\cone,thick] (-4*\w,0) -- (4*\w,0);

\newcommand{\htwoy}{0.5}
\newcommand{\htwox}{0.5*(1-\htwoy)}
\foreach \x in {-4,-3,-2,-1,0,1,2,3}
{
	
	\draw[color=\ctwo] ($\htwox*(\x*\w,0)+\htwox*(\x*\w+\w,0)+\htwoy*(0,\h)$) -- (\x*\w,0);
	\draw[color=\ctwo] ($\htwox*(\x*\w,0)+\htwox*(\x*\w+\w,0)+\htwoy*(0,\h)$) -- (\x*\w+\w,0);
	\draw[color=\ctwo] ($\htwox*(\x*\w,0)+\htwox*(\x*\w+\w,0)+\htwoy*(0,\h)$) -- (0,\h);

	\draw[color=\ctwo] ($\htwox*(\x*\w,0)+\htwox*(\x*\w+\w,0)+\htwoy*(0,-\h)$) -- (\x*\w,0);
	\draw[color=\ctwo] ($\htwox*(\x*\w,0)+\htwox*(\x*\w+\w,0)+\htwoy*(0,-\h)$) -- (\x*\w+\w,0);
	\draw[color=\ctwo] ($\htwox*(\x*\w,0)+\htwox*(\x*\w+\w,0)+\htwoy*(0,-\h)$) -- (0,-\h);
}

\foreach \x in {-4,-3,-2,-1,0,1,2}
{
	\draw[color=\cthree] ($\htwox*(\x*\w,0)+\htwox*(\x*\w+\w,0)+\htwoy*(0,\h)$) -- ($\htwox*(\x*\w+\w,0)+\htwox*(\x*\w+2*\w,0)+\htwoy*(0,\h)$);
	\draw[color=\cthree] ($\htwox*(\x*\w,0)+\htwox*(\x*\w+\w,0)+\htwoy*(0,-\h)$) -- ($\htwox*(\x*\w+\w,0)+\htwox*(\x*\w+2*\w,0)+\htwoy*(0,-\h)$);
}

\newcommand{\hfour}{0.25}
\foreach \x in {-4,-3,-2,-1,0,1,2,3}
{
	
	\draw[color=\cfour] (\x*\w,0) -- (\x*\w+0.5*\w,\hfour) -- (\x*\w+\w,0);
	\draw[color=\cfour] ($\htwox*(\x*\w,0)+\htwox*(\x*\w+\w,0)+\htwoy*(0,\h)$) -- (\x*\w+0.5*\w,\hfour) -- ($\htwox*(\x*\w,0)+\htwox*(\x*\w+\w,0)+\htwoy*(0,-\h)$);
}

\newcommand{\hfivey}{0.4}
\newcommand{\hfivex}{0.5*(1-\hfivey)}
\newcommand{\devfive}{0.15}


\coordinate (A) at ($\hfivex*\htwox*(-4*\w,0)+\hfivex*\htwox*(-4*\w+\w,0)+\hfivex*\htwoy*(0,\h)+\hfivex*\htwox*(-4*\w+\w,0)+\hfivex*\htwox*(-4*\w+2*\w,0)+\hfivex*\htwoy*(0,\h)+\hfivey*(-4*\w+\w,0)-(\devfive,0)$);
\filldraw[color=\cfive] (A) circle (2pt);

\coordinate (B) at ($\hfivex*\htwox*(-3*\w,0)+\hfivex*\htwox*(-3*\w+\w,0)+\hfivex*\htwoy*(0,\h)+\hfivex*\htwox*(-3*\w+\w,0)+\hfivex*\htwox*(-3*\w+2*\w,0)+\hfivex*\htwoy*(0,\h)+\hfivey*(-3*\w+\w,0)-(\devfive,0)$);
\filldraw[color=\cfive] (B) circle (2pt); 

\coordinate (C) at ($\hfivex*\htwox*(-2*\w,0)+\hfivex*\htwox*(-2*\w+\w,0)+\hfivex*\htwoy*(0,-\h)+\hfivex*\htwox*(-2*\w+\w,0)+\hfivex*\htwox*(-2*\w+2*\w,0)+\hfivex*\htwoy*(0,-\h)+\hfivey*(-2*\w+\w,0)-(\devfive,0)$);
\filldraw[color=\cfive] (C) circle (2pt); 

\coordinate (D) at ($\hfivex*\htwox*(-1*\w,0)+\hfivex*\htwox*(-1*\w+\w,0)+\hfivex*\htwoy*(0,-\h)+\hfivex*\htwox*(-1*\w+\w,0)+\hfivex*\htwox*(-1*\w+2*\w,0)+\hfivex*\htwoy*(0,-\h)+\hfivey*(-1*\w+\w,0)-(\devfive,0)$);
\filldraw[color=\cfive] (D) circle (2pt); 

\coordinate (E) at ($\hfivex*\htwox*(0*\w,0)+\hfivex*\htwox*(0*\w+\w,0)+\hfivex*\htwoy*(0,\h)+\hfivex*\htwox*(0*\w+\w,0)+\hfivex*\htwox*(0*\w+2*\w,0)+\hfivex*\htwoy*(0,\h)+\hfivey*(0*\w+\w,0)+(\devfive,0)$);
\filldraw[color=\cfive] (E) circle (2pt); 

\coordinate (F) at ($\hfivex*\htwox*(1*\w,0)+\hfivex*\htwox*(1*\w+\w,0)+\hfivex*\htwoy*(0,\h)+\hfivex*\htwox*(1*\w+\w,0)+\hfivex*\htwox*(1*\w+2*\w,0)+\hfivex*\htwoy*(0,\h)+\hfivey*(1*\w+\w,0)+(\devfive,0)$);
\filldraw[color=\cfive] (F) circle (2pt); 

\coordinate (G) at ($\hfivex*\htwox*(2*\w,0)+\hfivex*\htwox*(2*\w+\w,0)+\hfivex*\htwoy*(0,-\h)+\hfivex*\htwox*(2*\w+\w,0)+\hfivex*\htwox*(2*\w+2*\w,0)+\hfivex*\htwoy*(0,-\h)+\hfivey*(2*\w+\w,0)+(\devfive,0)$);
\filldraw[color=\cfive] (G) circle (2pt);

\foreach \x/\a in {-4/A,-3/B,0/E,1/F}
{
	\draw[color=\cfive] (\a) -- (\x*\w+\w,0);
	\draw[color=\cfive] (\a) -- ($\htwox*(\x*\w,0)+\htwox*(\x*\w+\w,0)+\htwoy*(0,\h)$);
	\draw[color=\cfive] (\a) -- ($\htwox*(\x*\w+\w,0)+\htwox*(\x*\w+2*\w,0)+\htwoy*(0,\h)$);
} 

\foreach \x/\a in {-2/C,-1/D,2/G}
{
	\draw[color=\cfive] (\a) -- (\x*\w+\w,0);
	\draw[color=\cfive] (\a) -- ($\htwox*(\x*\w,0)+\htwox*(\x*\w+\w,0)+\htwoy*(0,-\h)$);
	\draw[color=\cfive] (\a) -- ($\htwox*(\x*\w+\w,0)+\htwox*(\x*\w+2*\w,0)+\htwoy*(0,-\h)$);
}

\foreach \x/\a in {-4/A,-3/B,-2/C,-1/D,0/E,1/F,2/G}
{
	\draw[color=\csix] (\x*\w+0.5*\w,\hfour) -- (\a) -- (\x*\w+\w+0.5*\w,\hfour); 
}

\filldraw[color=\cone] (0,\h) circle (2pt);
\filldraw[color=\cone] (0,-\h) circle (2pt);

\foreach \x in {-4,-3,-2,-1,0,1,2,3,4}
{
	\filldraw[color=\cone,thick] (\x*\w,0) circle (2pt);
}

\foreach \x in {-4,-3,-2,-1,0,1,2,3}
{
	\filldraw[color=\ctwo] ($\htwox*(\x*\w,0)+\htwox*(\x*\w+\w,0)+\htwoy*(0,\h)$) circle (2pt);
	\filldraw[color=\ctwo] ($\htwox*(\x*\w,0)+\htwox*(\x*\w+\w,0)+\htwoy*(0,-\h)$) circle (2pt);
}


\foreach \x in {-4,-3,-2,-1,0,1,2,3}
{
	\filldraw[color=\cfour] (\x*\w+0.5*\w,\hfour) circle (2pt);
	}

\filldraw[color=\cfive] (A) circle (2pt);
\filldraw[color=\cfive] (B) circle (2pt); 
\filldraw[color=\cfive] (C) circle (2pt); 
\filldraw[color=\cfive] (D) circle (2pt); 
\filldraw[color=\cfive] (E) circle (2pt); 
\filldraw[color=\cfive] (F) circle (2pt); 
\filldraw[color=\cfive] (G) circle (2pt);

\end{scope}

\end{tikzpicture}}
\caption{\emph{Left:} 1-skeleton after stages \eone--\ethree. \emph{Right:} Final 1-skeleton. We removed the \tone{axis} from the picture and cut complex along a \tone{meridian}, removing some edges and vertices around this cutting.}
\label{fig:konstrukcja}
\end{figure}
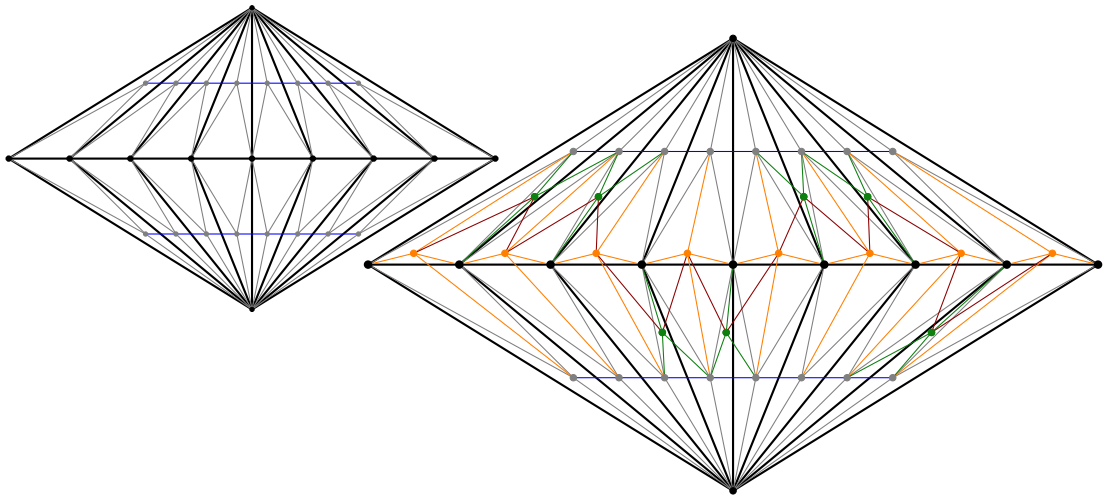

\smallskip	
\emph{Construction.} The construction is divided into 6 stages \eone--\esix.
To each stage we associate a colour, which
will be used in the pictures. The simplices created in the stage (s$i$) will be coloured in the colour associated to this stage and will be called (s$i$)-simplices (such a convention is extended to subcomplexes). 
Note that not all the edges of a (s$i$)-simplex are necessarily (s$i$)-edges, only those that do not belong to any (s$j$)-simplex for some $j<i$.
Each stage
can be subdivided into a sequence of steps, each consisting in gluing some 3-simplex to a 3-dimensional complex $\triang$ along one or two neighbouring faces contained in the boundary $\partial\triang$. Thus, at each step $\triang$ is a triangulation of the disk $D^3$ having all of its vertices on the boundary $\partial D^3$.
We call a subcomplex $\sigma$ \emph{external} if it is contained in the boundary $\partial\triang$,  \emph{internal} otherwise.
See Figure \ref{fig:konstrukcja}. The schematic view of the complex that is present there will be used throughout the whole proof.
\begin{enumerate}
\item[\eone] 
Take a \tone{simplicial join of an 8-cycle $C_8$ 
with an edge $E$}. 
The edge $E$ is called the \tone{axis}, the cycle $C_8$ is the \tone{equator}, the endpoints of the axis are the north and south pole. The remaining part consists of 
eight 2-paths 
with endpoints in the poles; we call them meridians.
	
\item[\etwo]  
Glue to each external \eone-face a \ttwo{3-simplex} (by one of its faces).
		
\item[\ethree] 
For each pair of \etwo-3-simplices that share an edge $e$ of a \tone{meridian}, glue a \tthree{3-simplex} 
with one of its faces glued to 
one of the external faces containing $e$, 
and other face glued to the other of the external faces containing $e$.

\item[\efour] 
For each pair of external \etwo-2-simplices $\ttwo{\Delta_1},\ttwo{\Delta_2}$ that share an edge of the \tone{equator},
add a simplicial cone over $\Delta_1\cup\Delta_2$. 
	
\item[\efive]  
We glue \tfive{3-simplices} to some of the external faces of \ethree-3-simplices, that intersect the \tone{equator}, in the following way.
Denote by $\tthree{N_1},\ldots,\tthree{N_8}$ ($\tthree{S_1},\ldots,\tthree{S_8}$) the faces described above that are above (below) the \tone{equator} (in cyclic order) in such a way that the face $\tthree{N_i}$ shares an edge with the face $\tthree{S_i}$. We glue a \tfive{3-simplex} to each of the faces $\tthree{N_1},\tthree{N_2},\tthree{S_3},\tthree{S_4},\tthree{N_5},\tthree{N_6},\tthree{S_7},\tthree{S_8}$.
	
\item[\esix] 
For each pair of 3-simplices \tfour{$\Delta_4$}, \tfive{$\Delta_5$} such that $\Delta_i$ is a (s$i$)-3-simplex, if \tfour{$\Delta_4$} and \tfive{$\Delta_5$} share an \etwo-edge $e$, we glue a \tsix{3-simplex} in such a way that one of its faces is glued to the external face of the \efour-simplex \tfour{$\Delta_4$} that contains the \etwo-edge \ttwo{$e$}, and 
another
face is glued to the external face of the \efive-simplex \tfive{$\Delta_5$} that contains the \etwo-edge \ttwo{$e$}.
Note that 
the newly added edges form the pattern
\tsix{$\mathsf{{}_W\!{}^M\!{}_W\!{}^M}$}.
\end{enumerate}

Henceforth we will denote the constructed triangulation by $\triang$ and use the notions of internal and external with respect to $\triang$.
In particular, all \eone-edges are internal, all \ethree-, \efour-, \efive- and \esix-edges are external and there are both internal and external \etwo-edges.

\smallskip
\emph{Non-planarity.} 
The complex $\triang$ contains a subcomplex 
that is
an edge subdivision $\Gamma$ of $\katt$: take all \esix-edges (i.e. \tsix{$\mathsf{{}_W\!{}^M\!{}_W\!{}^M}$}), the \tone{axis}, the two external 2-paths 
that join the south pole with the middle \efour-vertices of the fragments \tsix{$\mathsf{M}$} and the two external 2-paths 
that join the north pole with the middle \efour-vertices of the fragments \tsix{$\mathsf{W}$}.
One can check that $\Gamma$ is a full subcomplex of $\triang$ and is an edge subdivision of the graph $\katt$ such that the axis of $\triang$ corresponds to one of the edges of $\katt$
and for each of the remaining edges of $\katt$ there 
is a corresponding
path of length at least 2 in $\Gamma$. 
Therefore, in view of Remark \ref{u:tnieplan}, the complex $\triang$ is \sg-non-planar.

\begin{figure}
	\centerline{\input{fighiperbolicznosc.tex}}
\centering
\caption{The four remaining cases of checking the no-$\triangle$ and no-$\square$ conditions for $\triang$. Symbols \wiea{} and \wieb{} denote the endpoints of an edge potentially belonging to some 
empty triangle or some empty square,
symbols \sasa{} and \sasb{} denote the neighbours of \wiea{} and \wieb{}, respectively. It suffices to check that each vertex denoted by both \sasa{} and \sasb{} spans a 2-simplex together with \wiea{} and \wieb{}, and that there is no edge with one endpoint marked only with \sasa{} and the other endpoint marked only with \sasb.}
\label{fig:hiperbolicznosc}
\end{figure}	

\smallskip
\emph{No-$\triangle$ and no-$\square$.} The no empty triangle condition (\emph{no-$\triangle$}) is a part of checking flagness of the complex $\triang$. It means that there is no 3-cycle 
in $\triang$ that does not span a 2-simplex.

We inductively check that after each stage \eone--\esix{} of the construction both no-$\triangle$ and no-$\square$
are satisfied. 
Stage \eone{} may be viewed as two consecutive applications of the simplicial cone operation
to
an 8-cycle. 
The 8-cycle has no $\triangle$ and no $\square$, the operation of taking a cone preserves these properties. 
Since the stages \etwo, \efive{} consist in gluing 3-simplices along single faces, they cannot introduce a $\triangle$ or a $\square$.
Stage \efour{} may be viewed as a composition of two substages.
The first consists of a sequence of gluings of 8 \efour-simplices above the equator, each along one face.
The second consists of 
gluing 8 \efour-simplices below the equator, each along two faces.
Therefore a potential new $\triangle$ or $\square$ contains some \efour-edge having an end in some of \etwo-vertices. 
Similarly, stages \ethree{} and \esix, give 3 types of edges (up to symmetry), that potentially may introduce a new
$\triangle$ or $\square$.
Figure \ref{fig:hiperbolicznosc} 
contains all 
the cases not yet considered
and its description finishes the proof of this part. 

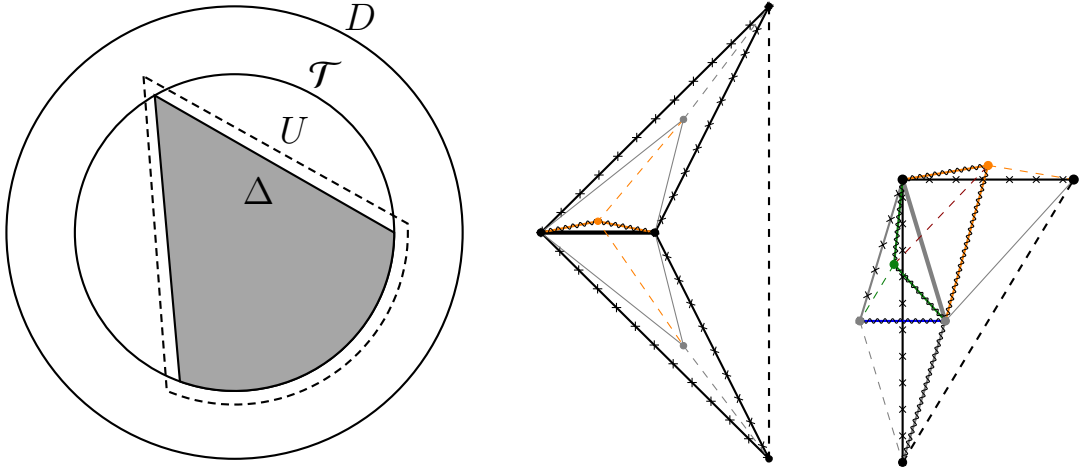
\begin{figure}
	\centerline{\begin{tikzpicture}
\begin{scope}[xshift=-200]
\begin{scope}[scale=0.6,rotate=90]
\newcommand{\rt}{3.5}
\newcommand{\ru}{3.8}
\newcommand{\rd}{5}

\draw[thick] (0,0) circle (\rt);
\draw[thick] (0,0) circle (\rd);

\filldraw[fill=gray!70,draw=black,thick] (30:\rt) -- (160:\rt) arc (160:270:\rt) -- cycle;

\draw[thick,densely dashed] (30:\ru*1.05) -- (157:\ru) arc (157:273:\ru) -- cycle;

\path (-30:0.3*\rt) node {{\Large $\Delta$}};
\path (-30:0.75*\rt) node {{\Large $U$}};
\path (-30:1.15*\rt) node {{\Large $\triang$}};
\path (-30:1.1*\rd) node {{\Large $D$}};

\end{scope}
\end{scope}

\newcommand{\h}{5}
\newcommand{\w}{3}

\newcommand{\htwoy}{0.5}
\newcommand{\htwox}{0.5*(1-\htwoy)}

\newcommand{\hfour}{0.25}

\newcommand{\hfivey}{0.4}
\newcommand{\hfivex}{0.5*(1-\hfivey)}
\newcommand{\devfive}{0.15}

\newcommand{\ocol}{violet}

\newcommand{\colinner}{black}
\newcommand{\colouter}{black}
\newcommand{\selesn}{1mm}
\newcommand{\ampsn}{0.3mm}

\begin{scope}[scale=0.6]
\renewcommand{\w}{2.5}
\foreach \x in {-2,-1}
{
	\draw[color=\cone,thick] (0,\h) -- (\x*\w,0) -- (0,-\h);
	\draw
	[color=\colinner,decorate,decoration={crosses}] 
	(0,\h) -- (\x*\w,0);
	\draw[color=\colinner,decorate,decoration={crosses}] 
	 (\x*\w,0) -- (0,-\h);
}
\draw[color=\cone,ultra thick] (-2*\w,0) -- (-1*\w,0);
\draw[color=\cone,thick,dashed] (0,\h) -- (0,-\h);

\foreach \x in {-2}
{
	\draw[color=\ctwo] ($\htwox*(\x*\w,0)+\htwox*(\x*\w+\w,0)+\htwoy*(0,\h)$) -- (\x*\w,0);
	\draw[color=\ctwo] ($\htwox*(\x*\w,0)+\htwox*(\x*\w+\w,0)+\htwoy*(0,\h)$) -- (\x*\w+\w,0);
	\draw[color=\ctwo,dashed] ($\htwox*(\x*\w,0)+\htwox*(\x*\w+\w,0)+\htwoy*(0,\h)$) -- (0,\h);
	
	\draw[color=\ctwo] ($\htwox*(\x*\w,0)+\htwox*(\x*\w+\w,0)+\htwoy*(0,-\h)$) -- (\x*\w,0);
	\draw[color=\ctwo] ($\htwox*(\x*\w,0)+\htwox*(\x*\w+\w,0)+\htwoy*(0,-\h)$) -- (\x*\w+\w,0);
	\draw[color=\ctwo,dashed] ($\htwox*(\x*\w,0)+\htwox*(\x*\w+\w,0)+\htwoy*(0,-\h)$) -- (0,-\h);
}

\foreach \x in {-2}
{
	\draw[color=\cfour,thick] (\x*\w,0) -- (\x*\w+0.5*\w,\hfour) -- (\x*\w+\w,0);
	\draw[color=\colouter,decorate,decoration={snake,segment length=\selesn,amplitude=\ampsn}] (\x*\w,0) -- (\x*\w+0.5*\w,\hfour);
	\draw[color=\colouter,decorate,decoration={snake,segment length=\selesn,amplitude=\ampsn}] (\x*\w+0.5*\w,\hfour) -- (\x*\w+\w,0);
	\draw[color=\cfour,dashed] ($\htwox*(\x*\w,0)+\htwox*(\x*\w+\w,0)+\htwoy*(0,\h)$) -- (\x*\w+0.5*\w,\hfour) -- ($\htwox*(\x*\w,0)+\htwox*(\x*\w+\w,0)+\htwoy*(0,-\h)$);
}

\filldraw[color=\cone] (0,\h) circle (2pt);
\filldraw[color=\cone] (0,-\h) circle (2pt);

\foreach \x in {-2,-1}
{
	\filldraw[color=\cone,thick] (\x*\w,0) circle (2pt);
}

\foreach \x in {-2}
{
	\filldraw[color=\ctwo] ($\htwox*(\x*\w,0)+\htwox*(\x*\w+\w,0)+\htwoy*(0,\h)$) circle (2pt);
	\filldraw[color=\ctwo] ($\htwox*(\x*\w,0)+\htwox*(\x*\w+\w,0)+\htwoy*(0,-\h)$) circle (2pt);
}

\foreach \x in {-2}
{
	\filldraw[color=\cfour] (\x*\w+0.5*\w,\hfour) circle (2pt);
}

\end{scope}

\begin{scope}[xshift=50,yshift=20]
\begin{scope}[scale=0.75]
\draw[color=\cone,thick] (0,-\h) -- (0*\w,0);
\draw[color=\colinner,decorate,decoration={crosses}] 
 (0,-\h) -- (0*\w,0);
\draw[color=\cone,thick,dashed] (0,-\h) -- (1*\w,0);

\draw[color=\cone,thick] (0*\w,0) -- (1*\w,0);
\draw[color=\colinner,decorate,decoration={crosses}] 
 (0*\w,0) -- (1*\w,0);

\foreach \x in {-1}
{
	\draw[color=\ctwo,thick] ($\htwox*(\x*\w,0)+\htwox*(\x*\w+\w,0)+\htwoy*(0,-\h)$) -- (\x*\w+\w,0);
	\draw[color=\colinner,decorate,decoration={crosses}] 
	 ($\htwox*(\x*\w,0)+\htwox*(\x*\w+\w,0)+\htwoy*(0,-\h)$) -- (\x*\w+\w,0);
	\draw[color=\ctwo,dashed] ($\htwox*(\x*\w,0)+\htwox*(\x*\w+\w,0)+\htwoy*(0,-\h)$) -- (0,-\h);
}

\foreach \x in {0}
{
	\draw[color=\ctwo,ultra thick] ($\htwox*(\x*\w,0)+\htwox*(\x*\w+\w,0)+\htwoy*(0,-\h)$) -- (\x*\w,0);
	\draw[color=\ctwo] ($\htwox*(\x*\w,0)+\htwox*(\x*\w+\w,0)+\htwoy*(0,-\h)$) -- (\x*\w+\w,0);
	\draw[color=\ctwo,thick] ($\htwox*(\x*\w,0)+\htwox*(\x*\w+\w,0)+\htwoy*(0,-\h)$) -- (0,-\h);
	\draw[color=\colouter,decorate,decoration={snake,segment length=\selesn,amplitude=\ampsn}] ($\htwox*(\x*\w,0)+\htwox*(\x*\w+\w,0)+\htwoy*(0,-\h)$) -- (0,-\h);
}

\foreach \x in {-1}
{
	\draw[color=\cthree,thick] ($\htwox*(\x*\w,0)+\htwox*(\x*\w+\w,0)+\htwoy*(0,-\h)$) -- ($\htwox*(\x*\w+\w,0)+\htwox*(\x*\w+2*\w,0)+\htwoy*(0,-\h)$);
	\draw[color=\colouter,decorate,decoration={snake,segment length=\selesn,amplitude=\ampsn}] ($\htwox*(\x*\w,0)+\htwox*(\x*\w+\w,0)+\htwoy*(0,-\h)$) -- ($\htwox*(\x*\w+\w,0)+\htwox*(\x*\w+2*\w,0)+\htwoy*(0,-\h)$);
}

\foreach \x in {0}
{
	\draw[color=\cfour,thick] (\x*\w,0) -- (\x*\w+0.5*\w,\hfour);
	\draw[color=\colouter,decorate,decoration={snake,segment length=\selesn,amplitude=\ampsn}] (\x*\w,0) -- (\x*\w+0.5*\w,\hfour); 
	\draw[color=\cfour,dashed] (\x*\w+0.5*\w,\hfour) -- (\x*\w+\w,0);
	\draw[color=\cfour,thick] (\x*\w+0.5*\w,\hfour) -- ($\htwox*(\x*\w,0)+\htwox*(\x*\w+\w,0)+\htwoy*(0,-\h)$);
	\draw[color=\colouter,decorate,decoration={snake,segment length=\selesn,amplitude=\ampsn}] (\x*\w+0.5*\w,\hfour) -- ($\htwox*(\x*\w,0)+\htwox*(\x*\w+\w,0)+\htwoy*(0,-\h)$);
}

\coordinate (D) at ($\hfivex*\htwox*(-1*\w,0)+\hfivex*\htwox*(-1*\w+\w,0)+\hfivex*\htwoy*(0,-\h)+\hfivex*\htwox*(-1*\w+\w,0)+\hfivex*\htwox*(-1*\w+2*\w,0)+\hfivex*\htwoy*(0,-\h)+\hfivey*(-1*\w+\w,0)-(\devfive,0)$);

\foreach \x/\a in {-1/D}
{
	\draw[color=\cfive,thick] (\a) -- (\x*\w+\w,0);
	\draw[color=\colouter,decorate,decoration={snake,segment length=\selesn,amplitude=\ampsn}] (\a) -- (\x*\w+\w,0);
	\draw[color=\cfive,dashed] (\a) -- ($\htwox*(\x*\w,0)+\htwox*(\x*\w+\w,0)+\htwoy*(0,-\h)$);
	\draw[color=\cfive,thick] (\a) -- ($\htwox*(\x*\w+\w,0)+\htwox*(\x*\w+2*\w,0)+\htwoy*(0,-\h)$);
	\draw[color=\colouter,decorate,decoration={snake,segment length=\selesn,amplitude=\ampsn}] (\a) -- ($\htwox*(\x*\w+\w,0)+\htwox*(\x*\w+2*\w,0)+\htwoy*(0,-\h)$); 
}

\foreach \x/\a in {-1/D}
{
	\draw[color=\csix,dashed] (\a) -- (\x*\w+\w+0.5*\w,\hfour); 
}

\filldraw[color=\cone] (0,-\h) circle (2pt);

\foreach \x in {0,1}
{
	\filldraw[color=\cone,thick] (\x*\w,0) circle (2pt);
}

\foreach \x in {-1,0}
{
	\filldraw[color=\ctwo] ($\htwox*(\x*\w,0)+\htwox*(\x*\w+\w,0)+\htwoy*(0,-\h)$) circle (2pt);
}


\foreach \x in {0}
{
	\filldraw[color=\cfour] (\x*\w+0.5*\w,\hfour) circle (2pt);
}

\filldraw[color=\cfive] (D) circle (2pt); 

\end{scope}
\end{scope}

\end{tikzpicture}}
\centering
\caption{\emph{Left:} situation in the proof of 1-dimensionality, for simplicity drawn in dimension 2. \emph{Middle, right:} the complex $\rys{e}$ in the remaining 2 cases from the proof of inseparability with 2-complexes. In the middle $e$ is an \eone-edge, on the right $e$ is an \etwo-edge. The bold lines mark the edge $e$, dashed lines mark the edges of the graph $\gra{e}$, the line
\raisebox{0.15em}{\protect\tikz{\protect\draw (0,0) -- (0.8,0);\protect\draw[decorate,decoration={snake,segment length=1mm,amplitude=0.3mm}] (0,0) -- (0.8,0)}}	(\raisebox{0.15em}{\protect\tikz{\protect\draw (0,0) -- (0.8,0);\protect\draw[decorate,decoration={crosses}] (0,0) -- (0.8,0)}}) 
marks the edges of the 2-paths that are external (internal) in all cases.}
\label{fig:niesep}
\end{figure}

\smallskip
\emph{1-dimensionality.} 
In the remaining part of the proof we consider $\triang$ more often as a topological space than as a simplicial complex.

Let $\Delta$ be a simplex of $\triang$. 
Embed $\triang$ as a 3-disk contained in the interior of a bigger 3-disk $D$ in a standard way with respect to the piecewise linear topology.
Let $U$ be a standard open normal neighbourhood of $\Delta$ in $D$.
See Figure \ref{fig:niesep}.
Then $U$ is homeomorphic 
to
the interior of the 3-disk and $\triang\setminus U$ is 
a deformation retract
of 
$\triang\setminus\Delta$.
The Mayer-Vietoris sequence for the sets $\triang\setminus U$ and $(D\setminus\intr\triang)\setminus U$ yields an exact sequence $H^k(D\setminus U)\to H^k(\triang\setminus U)\oplus H^k((D\setminus\intr\triang)\setminus U)\to H^k(\partial\triang\setminus U)$. 
Since $U$ is contractible and contained in the interior $\intr D$, the space $D\setminus U$ is a deformation retract of a 3-disk with one point removed, therefore the space $D\setminus U$ is homotopy equivalent to the boundary $\partial D\cong S^2$.
The space $(D\setminus\intr\triang)\setminus U$ is a deformation retract of $D\setminus\intr\triang$, which is homotopy equivalent to $S^2$.
Since $U\cap\partial\triang\neq\emptyset$, the space $\partial T\setminus U$ is a 2-manifold, whose each connected component has a non-empty boundary, so it is homotopy equivalent to some 1-complex.
Putting $k=2$ we get an exact sequence $\Z\to H^2(\triang\setminus U)\oplus\Z\to0$, so $H^2(\triang\setminus U)=0$.
For $k>2$ we have zeros on both sides, therefore $H^k(\triang\setminus U)=0$.

\smallskip
\emph{Inseparability.}
It is clear that $\triang$ is inseparable by a vertex or a pair of non-adjacent vertices.
The remaining part of the proof of inseparability relies on the following observation.

\begin{observation}\label{f:niesep}
Let $\sigma$ be a full subcomplex of $\,\triang$ of dimension at least 1 that is either a simplex or a suspension of a simplex.
If $\sigma$ separates $\triang$, then it also separates its boundary $\partial\triang$. 	
\end{observation}

\begin{proof}

Let $U$ be a standard open normal neighbourhood of $\sigma$ in $\mathcal{T}$, so that $U$ separates $\partial\triang$ iff $\sigma$ does. It is sufficient to show that each path component of $\triang\setminus U$ intersects $\partial\triang$ non-trivially. Take a path $\gamma$ that connects (in $\mathcal{T}$) a point $x\in \triang\setminus U$ with a point in $\partial\triang$. 
If it crosses the boundary $\partial U$, then, 
since $U$ is an open 3-disk whose boundary intersects the boundary
$\partial\triang$, we can get to the boundary $\partial\triang$ by a path contained in the boundary $\partial U$.
\end{proof}

The proof of inseparability consists in checking whether the intersections of appropriate complexes with the boundary $\partial\triang$ make it disconnected.
One can immediately prove inseparability by 1-simplices and suspensions of 0-simplices.

\smallskip
\emph{Inseparability by 2-complexes, no empty $\kacz$ and no $\kapi$.}
The no empty $\kacz$ and no $\kapi$ conditions are a part of checking flagness of $\triang$.
No empty $\kacz$ means that there is no subgraph isomorphic to $\kacz$ in the 1-skeleton $\triang^{(1)}$ that does not span a 3-simplex in $\triang$. No $\kapi$ means that there is no subgraph isomorphic to $\kapi$ in the 1-skeleton $\triang^{(1)}$.

The proof of inseparability by 2-complexes uses the following observation.

\begin{observation}\label{f:niesep2}
\begin{enumerate}[(i)]
\item A 2-simplex $\Delta$ disconnects the boundary $\partial\triang$ if and only if $\Delta\cap\partial\triang=\partial\Delta$.

\item If the suspension $\sigma$ of some 1-simplex disconnects the boundary $\partial\triang$, then either some 2-simplex of the complex $\sigma$ disconnects $\partial\triang$ or $\sigma\cap\partial\triang=\partial\sigma$. 

\end{enumerate}
\end{observation}
 
If a 2-simplex $\Delta$ disconnects $\partial\triang$, then by the above observation its intersection with the boundary $\partial\triang$ consists of all of its edges. 
We highlight the one that appeared in the earliest of the stages \eone--\esix{} (in case of having several such edges, highlight an arbitrary one). 
Otherwise, if some suspension $\sigma$ of a 1-simplex $e$ disconnects $\triang$, but none of its 2-simplices does, then highlight the edge $e$.
The method of the proof
is as follows. We consider each 1-simplex $e$ of the complex $\triang$ and check that it was not highlighted in the above procedure. 
More precisely, if $e$ is internal, then we check the condition (\en1): there are no two 2-paths 
contained in the boundary $\partial\triang$ with endpoints in 
$e$ such that their middle vertices do not span an edge (recall fullness in the definition of inseparability). If $e$ is external, we check the condition (\en2): there is no 2-path $\gamma$
in $\partial\triang$
with endpoints in 
$e$ 
consisting
of edges from stages not earlier than the one where $e$ 
first appears, such that
$\gamma\cup e$ does not span an external face.

Consider a subgraph of $\triang^{(1)}$ isomorphic to $\kacz$ or $\kapi$. Highlight one of its edges $e$ that 
appears
in the earliest of the stages \eone--\esix. Note that in the case of $\kapi$ 
the complex $\triang$ contains
the 1-skeleton of a join of the edge $e$ with a 3-cycle.

Now 
for each edge $e$ we define
the complex $\rys{e}$, that will 
allow us to decide
whether $e$ has been highlighted due to one of the four above reasons.
If $e$ is an internal edge, then we define $\rys{e}$ to be the full subcomplex of $\triang$ spanned by all 2-paths 
that connect the endpoints of the edge $e$.
If $e$ is external, then $\rys{e}$ is the full subcomplex of $\triang$ spanned by all 2-paths 
that connect the endpoints of the edge $e$ and 
consist
of edges that were created in the stages 
not earlier than the stage when the edge $e$ was created.
For each edge $e$ we define the graph $\gra{e}$ as the induced subgraph of the 1-skeleton $\rys{e}^{(1)}$ spanned by the vertices that are not the endpoints of the edge $e$ (i.e. the middle points of the appropriate 2-paths).
In order to check that there are no empty $\kacz$, we check the condition (\ky1): for each edge $e'$ of the graph $\gra{e}$ the complex $\rys{e}$ has the simplicial join of $e$ with $e'$ as one of its 3-simplices. 
In order to check that there is no $\kapi$, we check the condition (\ky2): there is no 3-cycle 
in the graph $\gra{e}$.
We split 
the proof
in cases depending on 
the stage in which $e$ first appears.
Our goal is to describe $\rys{e}$ and $\gra{e}$ 
in sufficient detail,
so that one can easily check the conditions (\ky1), (\ky2) and the appropriate one of (\en1), (\en2).
\begin{enumerate}
\item[\eone] 
There are
3 cases.
If \tone{$e$} is the \tone{axis}, the complex $\rys{e}$ is the complex created in the stage \eone. 
If \tone{$e$} is an \tone{edge of a meridian}, the complex $\rys{e}$ consists of 2 \etwo-paths with their middle vertices spanning an \ethree-edge, and 3 \eone-paths (which are internal). The graph $\gra{e}$ is a 5-cycle.
The case of \tone{$e$} being an \tone{edge of the equator} is shown in Figure \ref{fig:niesep}. The path consisting of 2 \efour-edges is external and at most one of the paths consisting of \etwo-edges is external, but the middle vertices of both of these paths are neighbours of the middle vertex of the path consisting of \efour-edges. 

\item[\etwo] There are 3 cases.
If \ttwo{$e$} contains a \tone{pole} or is an external edge containing some vertex of the \tone{equator}, then $\rys{e}$ consists of two neighbouring \ethree-2-simplices or an \ethree-2-simplex and an \efour-2-simplex, respectively.
If \ttwo{$e$} contains a vertex of the \tone{equator} and is internal, see Figure \ref{fig:niesep}. 

\item[\ethree] There are 2 cases. The complex $\rys{e}$ 
is either empty
or consists of an \efive-2-simplex containing the edge \tthree{$e$}.

\item[\efour] If \tfour{$e$} is an \efour-edge with 
endpoints not on the \tone{equator},
we have 3 cases: the complex $\rys{e}$ 
is either empty,
or consists of a single \esix-2-simplex, 
or consists of two \esix-2-simplices. 
If \tfour{$e$} is an \efour-edge with one endpoint on the \tone{equator}, we have 2 cases, in each of them the complex $\rys{e}$ consists of a single \esix-2-simplex.

\item[\efive] All \efive-edges are external and there is no 2-simplex whose edges are only \efive- or 
\esix-edges, therefore the complex $\rys{e}$ 
is empty.

\item[\esix] As above, the complex $\rys{e}$ is empty.

\end{enumerate}

\smallskip
\emph{Inseparability by 3-complexes.}
The proof of this part uses the following observation that enables us to reduce the problem to the already considered case of inseparability by 2-complexes.

\begin{observation}\label{f:niesep3}
Let $\Delta$ be a 2-face of a 3-simplex $\sigma$ of $\triang$. 
Define (a subspace) $\tra{\sigma}{\Delta}:=(\sigma\cap\partial\triang)\setminus\Delta$.
Assume that the
intersection
$\overline{\tra{\sigma}{\Delta}}\cap\Delta$ is either empty or is a path.
Then: 

\begin{enumerate}[(i)]
	\item if a 3-simplex $\sigma$ disconnects the boundary $\partial\triang$, then the 2-simplex $\Delta$ disconnects the boundary $\partial\triang$, 
	\label{f:niesep31}
	
	\item if $\sigma'$ is a 3-simplex such that the complex $\sigma\cup\sigma'$ is a full subcomplex of $\triang$ that is a suspension of the face $\Delta$   
	and
	the complex $\sigma\cup\sigma'$ disconnects the boundary $\partial\triang$, then the 3-simplex $\sigma'$ disconnects the boundary $\partial\triang$.
	\label{f:niesep32}
	
\end{enumerate}	
\end{observation}

\begin{proof}
The complex $\overline{\tra{\sigma}{\Delta}}$ is a cone over $\overline{\tra{\sigma}{\Delta}}\cap\Delta$.
If the latter is empty, then the former is a single vertex and the observation follows.
Otherwise $\overline{\tra{\sigma}{\Delta}}$ is a deformation retract of $\overline{\tra{\sigma}{\Delta}}\cap\Delta$, which gives a homotopy equivalence between
$\partial\triang\setminus\Delta$
and $\partial\triang\setminus\sigma$ and, in case
\ref{f:niesep32},
a homotopy equivalence between $\partial\triang\setminus\sigma$ and $\partial\triang\setminus(\sigma\cup\sigma')$. 	
\end{proof}

The proof consists in assigning to each 3-simplex some of its 2-faces $\Delta_{\sigma}$ and then assigning to each internal 2-simplex $\Delta$ a 3-simplex $\sigma_{\Delta}$ in such a way that the pairs $(\sigma,\Delta_{\sigma})$ and $(\sigma_{\Delta},\Delta)$ satisfy the assumptions of Observation \ref{f:niesep3}.
Then the inseparability by 3-complexes follows in 
the following
way. If $\sigma$ is a 3-simplex of $\triang$ that disconnects the boundary $\partial\triang$, then by Observation \ref{f:niesep3}\ref{f:niesep31} 
its 2-dimensional face $\Delta_{\sigma}$ disconnects the boundary $\partial\triang$. We have 
ruled that out 
in the previous part of the proof.
If the suspension  $\tau$ of a 2-simplex $\Delta$ is a subcomplex of $\triang$ disconnecting its boundary $\partial\triang$, then the 3-simplex $\sigma_{\Delta}$ is one of simplices of $\tau$ and, denoting by $\sigma'$ the other one, 
it follows by 
Observation \ref{f:niesep3}\ref{f:niesep32} that
the 3-simplex $\sigma'$ disconnects the boundary $\partial\triang$. We have just 
ruled this out
above.
Below we choose some pairs $(\sigma,\Delta)$ as above and describe the spaces $\overline{\tra{\sigma}{\Delta}}$. 
We do this in the order given by the stage in which the 
simplex
$\sigma$ was constructed.
We leave checking that the described pairs satisfy the assumptions of Observation \ref{f:niesep3} and verifying that each internal face has been paired with some 3-simplex to the reader.

\begin{enumerate}
\item[\eone] The intersection of any \eone-3-simplex \tone{$\sigma$} with the boundary $\partial\triang$ consists of 4 points, in particular $\overline{\tra{\mone{\sigma}}{\mone{\Delta}}}$ is a single point for any \eone-face \tone{$\Delta$} of the \eone-simplex \tone{$\sigma$}.   

\item[\etwo] For any \etwo-3-simplex \ttwo{$\sigma$} and \efour-3-simplex that share an \etwo-face \ttwo{$\Delta$},
$\overline{\tra{\ttwo{\sigma}}{\ttwo{\Delta}}}$ is a single edge.

\item[\ethree] Consider an \ethree-3-simplex \tthree{$\sigma$}. Its \ethree-face \tthree{$\Delta$} that
intersects
the \tone{equator} non-trivially has the property that $\overline{\tra{\tthree{\sigma}}{\tthree{\Delta}}}$ is a single face. Consider any \etwo-face \ttwo{$\Delta'$}, which is shared by the \ethree-3-simplex \tthree{$\sigma$} and an \etwo-3-simplex. If \tthree{$\sigma$} has 2 external faces, then $\overline{\tra{\tthree{\sigma}}{\ttwo{\Delta'}}}$ 
is the
suspension of an edge. If \tthree{$\sigma$} has one external face, $\overline{\tra{\tthree{\sigma}}{\ttwo{\Delta'}}}$ is a single face.  

\item[\efour] Let \tfour{$\Delta$} be an \efour-face of an \efour-3-simplex \tfour{$\sigma$}, which \tfour{$\sigma$} shares with an \efour-3-simplex. The simplex \tfour{$\sigma$} can lie in $\triang$ in 3 ways, 
and
$\overline{\tra{\tfour{\sigma}}{\tfour{\Delta}}}$ is either a  single edge or a single face or a suspension of an edge.

\item[\efive] Let \tfive{$\sigma$} be a  \efive-3-simplex and \tthree{$\Delta$} its \ethree-face. Then $\overline{\tra{\tfive{\sigma}}{\tthree{\Delta}}}$ is a single face.
 
\item[\esix] For each \esix-3-simplex \tsix{$\sigma$} and its face $\Delta$ that is internal, 
$\overline{\tra{\tsix{\sigma}}{\Delta}}$ is a suspension of an edge. 
\end{enumerate}

\smallskip
\emph{Flagness.} Consider an $n$-clique subgraph of the 1-skeleton $\triang^{(1)}$, that does not span a simplex. We have proved that $n\geq 5$. Therefore the 1-skeleton $\triang^{(1)}$ contains a  subgraph isomorphic to $\kapi$, which has also been 
ruled out.
\end{proof}

\begin{remark}
	A slight modification of the triangulation $\triang$ gives infinitely many triangulations of the disk $D^3$ that give a Menger curve boundary.
	Indeed, fix $n\geq2$. In stage \eone{} take a join of a $4n$-cycle 
	with an edge, leave the stages \etwo--\efour{} unchanged, and execute the stages \efive{} and \esix{} so that the pattern \tsix{$\mathsf{{}_W\!{}^M}$}, that exists around the equator, repeats $n$ times (note that for $n=2$ we get $\triang$).
	The proof that such a triangulation satisfies the desired properties is almost the same as for $\triang$.
	For example, the proof of inseparability uses that fact that the equator is of length greater than 4 and 
	analyses
	the shape of some small fragments of the considered triangulation that do not depend on $n$.
	
\end{remark}

\bibliographystyle{alpha}
\bibliography{matbibDD}

\end{document}